\documentclass[review,onefignum,onetabnum]{siamart190516noline}
\usepackage{graphicx, subfigure}
\usepackage{amsmath}
\usepackage{amssymb}
\usepackage{verbatim}
\usepackage{enumerate}
\usepackage{extarrows}

\crefname{assumption}{Assumption}{Assumptions}
\newcommand{\bit}{\begin{itemize}}
	\newcommand{\eit}{\end{itemize}}
\newcommand{\be}{\begin{equation}}
\newcommand{\en}{\end{equation}}

\newcommand{\diag }{{\rm diag}}
\newcommand{\Diag }{{\rm Diag}}
\newcommand{\RE}{{\rm Re}}
\newcommand{\IM}{{\rm Im}}
\newcommand{\TR}{{\rm Tr}}

\usepackage{multirow}
\usepackage{adjustbox}
\usepackage[figuresleft]{rotating}
\usepackage{makecell}
\usepackage{algorithmic}
\usepackage[most]{tcolorbox}

\usepackage[utf8]{inputenc}

\usepackage{pgf,tikz,pgfplots}
\pgfplotsset{compat=1.12}
\usepackage{mathrsfs}
\usetikzlibrary{arrows}
\usepackage{extarrows}

\usepackage{lipsum}
\usepackage{amsfonts}
\usepackage{graphicx}
\usepackage{epstopdf}
\usepackage{algorithmic}
\ifpdf
  \DeclareGraphicsExtensions{.eps,.pdf,.png,.jpg}
\else
  \DeclareGraphicsExtensions{.eps}
\fi

\newsiamremark{remark}{Remark}
\newsiamremark{hypothesis}{Hypothesis}
\crefname{hypothesis}{Hypothesis}{Hypotheses}
\newsiamthm{claim}{Claim}

\headers{An Efficient Algorithm for Large-Scale MIMO Detection}{P.-F. Zhao, Q.-N. Li, W.-K. Chen, and Y.-F. Liu}

\title{An Efficient Quadratic Programming Relaxation Based Algorithm for Large-Scale MIMO Detection\thanks{Submitted on June 22, 2020, first revised on November 23, 2020, accepted on March 4, 2021.
\funding{The work of Qing-Na Li was supported by the National Natural Science Foundation of China (NSFC) 12071032, 11671036. The work of Ya-Feng Liu was supported in part by NSFC under Grant 12022116, Grant 12021001 , Grant 11688101, and Grant 11991021. The work of Wei-Kun Chen was supported in part by by Beijing Institute of Technology Research Fund Program for Young Scholars (Nos. 3170011181905 and 3170011182012).}}}

\author{Ping-Fan Zhao\thanks{School of Mathematics and Statistics, Beijing Institute of Technology, Beijing
  (\email{pfzhao@bit.edu.cn}).}
\and Qing-Na Li\thanks{Qing-Na Li is the corresponding author. School of Mathematics and Statistics/Beijing Key Laboratory on MCAACI, Beijing Institute of Technology, Beijing 
	(\email{qnl@bit.edu.cn}, \email{chenweikun@bit.edu.cn}).}
\and Wei-Kun Chen\footnotemark[3]
\and Ya-Feng Liu\thanks{State Key Laboratory of Scientific and Engineering Computing, Institute of Computational Mathematics and Scientific/Engineering Computing, Academy of Mathematics and Systems Science, Chinese Academy of Sciences, Beijing 
	(\email{yafliu@lsec.cc.ac.cn}).}}

\usepackage{amsopn}

\ifpdf
\hypersetup{
	pdftitle={An Efficient Quadratic Programming Relaxation Based Algorithm for Large-Scale MIMO Detection},
	pdfauthor={P.-F. Zhao, Q.-N. Li, W.-K. Chen, and Y.-F. Liu}
}
\fi

\begin{document}
	\definecolor{xdxdff}{rgb}{0.49019607843137253,0.49019607843137253,1}
	\maketitle	
	
	\begin{abstract}
		Multiple-input multiple-output (MIMO) detection is a fundamental problem in wireless communications and it is strongly NP-hard in general. Massive MIMO has been recognized as a key technology in the fifth generation (5G) and beyond communication networks, which on one hand can significantly improve the communication performance, and on the other hand poses new challenges of solving the corresponding optimization problems due to the large problem size.
		While various efficient algorithms such as semidefinite relaxation (SDR) based approaches have been proposed for solving the small-scale MIMO detection problem, they are not suitable to solve the large-scale MIMO detection problem due to their high computational complexities. In this paper, we propose an efficient sparse quadratic programming (SQP) relaxation based algorithm for solving the large-scale MIMO detection problem. In particular, we first reformulate the MIMO detection problem as an SQP problem. By dropping the sparse constraint, the resulting relaxation problem shares the same global minimizer with the SQP problem.  In sharp contrast to the SDRs for the MIMO detection problem, our relaxation does not contain any (positive semidefinite) matrix variable and the numbers of variables and constraints in our relaxation are significantly less than those in the SDRs, which makes it particularly suitable for the large-scale problem. Then we propose a projected Newton based quadratic penalty method to solve the relaxation problem, which is guaranteed to converge to the vector of transmitted signals under reasonable conditions. By extensive numerical experiments, when applied to solve small-scale problems, the proposed algorithm is demonstrated to be competitive with the state-of-the-art approaches in terms of detection accuracy and solution efficiency; when applied to solve large-scale problems, the proposed algorithm achieves better detection performance than a recently proposed generalized power method. 
	\end{abstract}
	
	
	
	\section{Introduction}
	Multiple-input multiple-output (MIMO) detection is a fundamental problem in  modern  communications \cite{albreem2019massive, 1yang2015fifty}. 
	The input-output relationship of the MIMO channel is
	\begin{equation}\label{MIMO}
	r=H{x}^*+v,
	\end{equation}
	where
	$ r\in {\mathbb{C}}^m $ denotes the vector of received signals,	
	$ H\in {\mathbb{C}}^{m\times n} $ denotes an $ m\times n $ complex channel matrix (usually $ m\geqslant n $), $ {x}^*\in {\mathbb{C}}^n $ denotes the vector of transmitted signals, and $ v\in {\mathbb{C}}^m $ denotes an additive white circularly symmetric Gaussian noise. The goal of MIMO detection is to recover the transmitted signals $ x^* $ from the received signals $ r $ based on the channel information $ H $. We refer to \cite{jalden2006detection,1yang2015fifty} for a review of different formulations and approaches for MIMO detection and \cite{albreem2019massive} for the latest progress in MIMO detection.
	
	In this paper, we assume that  $x^*$ in \cref{MIMO} is modulated via the  $M $-Phase-Shift Keying ($ M $-PSK) modulation scheme with $ M\geqslant2 $. More exactly, each entry $ x^*_j $ of $ x^* $ belongs to a finite set:
	
	\begin{equation}\label{xj}
	x_j^*\in \mathcal{X}\triangleq\left \{\mathrm{exp}(\mathrm{i}\theta)\ \left|\ \theta=\dfrac{2(k-1)\pi}{M},\  k=1,\ldots,M\right. \right \},\ j=1, \ldots,n,
	\end{equation}
	where $ \mathrm{i} $ is the imaginary unit. The mathematical formulation for the MIMO detection problem is
	\begin{equation}
	\tag{{\rm{P}}}
	\label{P}
	\begin{aligned}
	& \underset{x\in {\mathbb{C}}^n}{\min}
	& & F(x)\triangleq{\|Hx-r\|}_2^2 \\
	& \ \text{s.t.}
	& & {\left| x_j\right| }^2=1,\ j=1, \ldots,n,\\
	&&& \arg(x_j)\in \mathcal{A}\triangleq\left\{ 0,\,\dfrac{2\pi}{M},\ldots,\,\dfrac{2(M-1)\pi}{M}\right\} ,\ j=1, \ldots,n,
	\end{aligned}
	\end{equation}
	where $ \|\cdot\|_2 $ denotes the Euclidean norm and $ \arg(\cdot) $ denotes the argument of the complex number.
	
	Let
	\begin{equation}\label{def-Q}
	Q=H^\dagger H\text{ and }c=-H^\dagger r ,
	\end{equation}
	where $ (\cdot)^\dagger  $ denotes the conjugate transpose.
	Then problem \cref{P} is equivalent to the following complex quadratic programming problem
	\begin{equation}
	\tag{{\rm{CQP}}}
	\label{CQP}
	\begin{aligned}
	& \underset{x\in {\mathbb{C}}^n}{\min}
	& & x^\dagger Qx+2\RE(c^\dagger x)\\
	& \ \text{s.t.}
	& & {\left| x_j\right| }^2=1,\ j=1, \ldots,n,\\
	&&& \arg(x_j)\in \mathcal{A},\ j=1, \ldots,n,
	\end{aligned}
	\end{equation}
	where $ \RE(\cdot) $ denotes the real part of the complex number.
	
	Various methods to tackle the MIMO detection problem can be summarized into several lines \cite[Figure 15]{1yang2015fifty}, including tree search \cite{o2000lattice,pohst1981on, xie1993joint}, lattice reduction (LR) \cite{lenstra1982factoring,zhou2013element}, and semidefinite relaxation (SDR) \cite{lu2017an,lu2020an,tan2001the, wai2011cheap}. The tree search based methods are the most popular detectors in the era of multi-antenna MIMO systems \cite{1yang2015fifty}. Taking the typical tree search based method, the sphere decoder (SD) algorithm \cite{o2000lattice}, as an example, it is regarded as the benchmark for globally solving the MIMO detection problem. However, both the expected and worst-case complexities of the SD algorithm are exponential \cite{jalden2005on,verdu1989computational}. The most popular LR algorithm is the Lenstra-Lenstra-Lov$ \acute{a} $sz (LLL) algorithm \cite{lenstra1982factoring}, whose worst-case computational complexity can be prohibitively high \cite{jalden2008worst, yao2003efficient}. Below we mainly review the SDR based approach, which is most related to this work.
	
	The SDR based approach was first proposed for a binary PSK (BPSK) modulated code division multiple access (CDMA) system \cite{tan2001the}. Then it was extended to the quadrature PSK (QPSK) scenario \cite{Lotter1998Space} and further to the high-order $M$-PSK scenario \cite{luo2003an, ma2004semidefinite}. In \cite{mobasher2007a}, a quadratic assignment problem formulation was proposed for problem \cref{P}, and a near-maximum-likelihood decoding algorithm was designed based on the resulting SDR. Other early SDR based approaches are summarized in \cite[Table IX]{1yang2015fifty}. 
	
	SDR based approaches generally perform very well for solving the MIMO detection problem. To understand the reason, various researches have been done and one line of researches is to identify conditions under which the SDRs are \emph{tight} \cite[Definition 1]{lu2019tightness}.
	For the case where $M=2$, So \cite{so2010probabilistic} proposed an SDR of problem \cref{P} and proved its tightness when the following condition
	\begin{equation}\label{Cond1}
	\lambda_{\min}(\RE(H^\dagger H))>\|\RE(H^\dagger v)\|_\infty
	\end{equation}
	is satisfied.  Here $ H$ and $v $ are defined in \cref{MIMO}, $ \lambda_{\min}(\cdot) $ denotes the smallest eigenvalue of a given matrix, and $ \|\cdot\|_\infty $ denotes
	the $ \ell_\infty $-norm. An open question proposed in \cite{so2010probabilistic} is that whether the (conventional) SDR is still tight under condition \cref{Cond1} for the case where $M\geqslant 3$. It was negatively answered in \cite{lu2019tightness}. In addition, Lu et al. in \cite{lu2019tightness} proposed an enhanced SDR (see \cref{ERSDR1} further ahead) by adding some valid inequalities and showed that under condition
	\begin{equation}\label{Cond2}
	\lambda_{\min}(H^\dagger H)\sin\left(\frac{\pi}{M}\right) >\|H^\dagger v\|_\infty,
	\end{equation}
	\cref{ERSDR1} is tight. In \cite{liu2019on}, the relations between different SDRs were further analyzed. In particular, it was proved that \cref{ERSDR1} and the SDR proposed in \cite{mobasher2007a} are equivalent, and as a result, the SDR proposed in \cite{mobasher2007a} is also tight under condition \cref{Cond2}. Other representative analysis results can be found in \cite{busari2018millimeter-wave, molisch2017hybrid, hestenes2021tightness}. 
	
	One key advantage of the SDR based approaches, compared to SD and LLL algorithms, is that the SDR admits polynomial-time algorithms. There are well developed solvers for solving the SDR, such as MOSEK \cite{mosek} and the latest SDPNAL+ \cite{sun2020SDPNAL, yang2015SDPNAL, zhao2009a, zhao2010a}. However, the numbers of variables and constraints in the SDRs are much larger than those in problem \cref{CQP}, and hence the SDR based approaches cannot be used to solve the \emph{large-scale} MIMO detection problem. On the other hand, it was predicted that the mobile data traffic will grow exponentially in 2017-2022 \cite{cisco2019cisco}, which calls for higher data rates, larger network capacity, higher spectral efficiency, higher energy efficiency, and better mobility \cite{albreem2019massive}. 
	Massive MIMO is a key and effective technology to meet the above requirements, where the base station (BS) is equipped with tens to hundreds of antennas, in contrast to the current BS equipped with only 4 to 8 antennas. A new challenge coming with the massive MIMO technology is the large problem size in signal processing and optimization. In particular, the MIMO detection problem of our interest in the massive MIMO setup is a \emph{large-scale} strongly NP-hard problem \cite{verdu1989computational}. As far as we know, there are very few works on the large-scale MIMO detection problem. One notable work is \cite{liu2017a}, which proposes a customized generalized power method (GPM) for solving the large-scale MIMO detection problem. The GPM directly solves problem \cref{P} and at each iteration, the algorithm takes a gradient descent step with an appropriate stepsize and projects the obtained point onto the (discrete) feasible set of problem \cref{P}.
	However, our experiments show that the performance of the GPM heavily depends on the choice of the initial point. Consequently, models and algorithms that can be generalized to the large-scale MIMO detection problem with satisfactory detection performance are still highly in need. 
	
	{\bf Contributions.} 
	The contributions of the paper are twofold. Firstly, we propose a sparse quadratic programming (SQP) formulation for the MIMO detection problem. We prove that, somewhat surprisingly, its relaxation obtained by dropping the sparse constraint is equivalent to the original formulation. Moreover, the relaxation formulation is able to recover the vector of transmitted
signals under condition \cref{Cond2}. Secondly, we present a projected Newton based quadratic penalty (PN-QP) method to solve the proposed (relaxation) formulation, which is demonstrated to be quite efficient in terms of detection accuracy and solution efficiency. Under reasonable assumptions on the channel matrix and noise, the sequence generated by PN-QP is guaranteed to converge to the vector of transmitted
signals.
	In particular, our extensive numerical results show that (i) compared to SD and MOSEK (for solving \cref{ERSDR1}), PN-QP is more efficient on massive MIMO detection; (ii) compared to GPM, PN-QP achieves significantly better detection performance than a recently proposed generalized power method. 
	
	Two key features of our proposed approach are highlighted as follows. Firstly, in sharp contrast to the matrix based SDRs, due to the vector based formulation for the MIMO detection problem, our relaxation is particularly suitable to deal with the large-scale MIMO detection problem. Secondly, by exploring the sparse structure of the optimal solution, the computational cost of PN-QP is significantly reduced. In particular, PN-QP is designed to identify the support set of the optimal solution rather than to find the solution itself, leading to a low computational cost.

	The rest of this paper is organized as follows.
    In \cref{sec-sparse QP}, we introduce different formulations for the MIMO detection problem, including the SQP formulation. In \cref{sec-tightness}, we discuss the relaxation problem and its properties. In \cref{sec-alg}, we present the PN-QP method and its convergence result.
    In \cref{sec-numerical}, we perform extensive numerical experiments to compare different algorithms for solving the MIMO detection problem. Finally, we conclude the paper in \cref{sec-conclusions}.
	
	We adopt the following standard notations in this paper.	
	Let $ \mathrm{i} $ denote the imaginary unit (satisfying $ \mathrm{i}^2=-1 $). 
	For a given complex vector $ x $, we use $ x_j $ to denote its $ j $-th entry, and $ |x_j| $ to denote the modulus of its $ j $-th entry.  Let $\|\cdot\|_2$ denote the $\ell_2$-norm for vectors and Frobenius norm for matrices. We use $ \|x\|_0 $ to denote the number of nonzero entries in vector $ x $.  Let $ \diag (X) $  denote the  vector formed by the diagonal elements in matrix $X$, and $\Diag (x)$ denote the diagonal matrix with the diagonal entries being vector $ x $. For matrices $ X_{11},\ldots,X_{nn}\in\mathbb{R}^{M\times M} $, we also use $ \Diag(X_{11},\ldots,X_{nn})\in\mathbb{R}^{nM\times nM} $ to denote the block-diagonal matrix whose $ (j,\,j) $-th block is $ X_{jj} $. For a complex matrix $C$, let $ \RE(C) $ and $ \IM(C) $  denote the real and imaginary parts of $ C $, respectively, and $ C^\dagger  $ and $ C^\top $ denote the conjugate transpose and transpose of $ C $, respectively. $C\succeq0 $ means $ C $ is positive semidefinite,  and $ \TR(C) $ denotes the trace of $ C $. Define the inner product for $ x,\,v\in\mathbb{C}^n $ as $ \langle x,\,v\rangle=\RE(x^\dagger v) $. For two Hermitian matrices $ A $ and $ B $,  the inner product is defined similarly as $ \langle A,\, B\rangle = \RE(\TR(A^\dagger B)) $.  Let $ \boldsymbol{e} $ be a vector of an appropriate length with all elements being one. 
	For a sequence $ \{x^k\} $, $ x^k\uparrow c $ and $ x^k\downarrow c $ mean that $ x^k $ tends to increasingly and decreasingly to a certain value $ c $, respectively.
	We use $\otimes$ to denote the Kronecker product. For $t\in \mathbb{R}^{nM} $, we assume that $t$ has the partition as $t=(\bar t_1^{\top},\ldots, \bar t_n^\top)^{\top}$, where $\bar t_j\in\mathbb{R}^M$ is the $j$-th block of $t$. Finally, the $k$-th entry in block $\bar t_j$ is denoted as $(\bar t_j)_k$.

	\section{Different Formulations for MIMO Detection}\label{sec-sparse QP}
	In this section, we introduce some formulations for the MIMO detection problem and discuss their properties. 
	
	Define
	\begin{equation}\label{YJIHE}
		\mathcal{Y}=\left\{ (\cos \theta_k,\,\sin\theta_k)\ \left|\ \theta_k=\frac{2(k-1)\pi}{M},\ k=1,\ldots,M\right. \right\} .
	\end{equation}
	Then, for each $ j=1,\ldots,n $, it is easy to see that $ x_j\in\mathcal{X} $ (defined in \cref{xj}) if and only if $ \left(\RE(x_j),\,\IM(x_j)\right) \in \mathcal{Y}. $
	The feasible points of $\mathcal{Y}$ for $M=4$ and $M=8$ are illustrated in \cref{figure-1}.
	\begin{figure}[htbp]
		\centering
		\subfigure[$ M=4 $.]{
			\begin{minipage}[t]{0.45\linewidth}
				\centering
				\begin{tikzpicture}[scale=0.9,line cap=round,line join=round,>=triangle 45,x=1cm,y=1cm]
				\begin{axis}[
				x=1cm,y=1cm,
				axis lines=middle,
				xmin=-2.5,
				xmax=2.5,
				ymin=-2.5168356882371397,
				ymax=2.6423023198306312,
				xtick={0,4,...,0},
				ytick={0,4,...,0},]
				\draw [line width=0.8pt] (0,0) circle (2cm);
				\draw (2,0) node[anchor=north west] {$1$};
				\begin{scriptsize}
				\draw [fill=xdxdff] (0,2) circle (3pt);
				\draw [fill=xdxdff] (-2,0) circle (3pt);
				\draw [fill=xdxdff] (0,-2) circle (3pt);
				\draw [fill=xdxdff] (2,0) circle (3pt);
				\end{scriptsize}
				\end{axis}
				\end{tikzpicture}
			\end{minipage}
		}
		\subfigure[$ M=8 $.]{
			\begin{minipage}[t]{0.45\linewidth}
				\centering
				\begin{tikzpicture}[scale=0.9,line cap=round,line join=round,>=triangle 45,x=1cm,y=1cm]
				\begin{axis}[
				x=1cm,y=1cm,
				axis lines=middle,
				xmin=-2.5,
				xmax=2.5,
				ymin=-2.5168356882371397,
				ymax=2.6423023198306312,
				xtick={0,4,...,0},
				ytick={0,4,...,0},]
				\clip(-3.3189305599459695,-2.5168356882371397) rectangle (3.4368637547331145,2.6423023198306312);
				\draw [line width=0.8pt] (0,0) circle (2cm);
				\draw (2,0) node[anchor=north west] {$1$};
				\begin{scriptsize}
				\draw [fill=xdxdff] (1.4142135623730951,1.414213562373095) circle (3pt);
				\draw [fill=xdxdff] (-1.414213562373095,1.4142135623730951) circle (3pt);
				\draw [fill=xdxdff] (-1.4142135623730954,-1.414213562373095) circle (3pt);
				\draw [fill=xdxdff] (1.4142135623730947,-1.4142135623730954) circle (3pt);
				\draw [fill=xdxdff] (0,2) circle (3pt);
				\draw [fill=xdxdff] (-2,0) circle (3pt);
				\draw [fill=xdxdff] (0,-2) circle (3pt);
				\draw [fill=xdxdff] (2,0) circle (3pt);
				\end{scriptsize}
				\end{axis}
				\end{tikzpicture}
			\end{minipage}
		}
		\centering
		\caption{An illustration of $\mathcal Y$.}
		\label{figure-1}
	\end{figure}
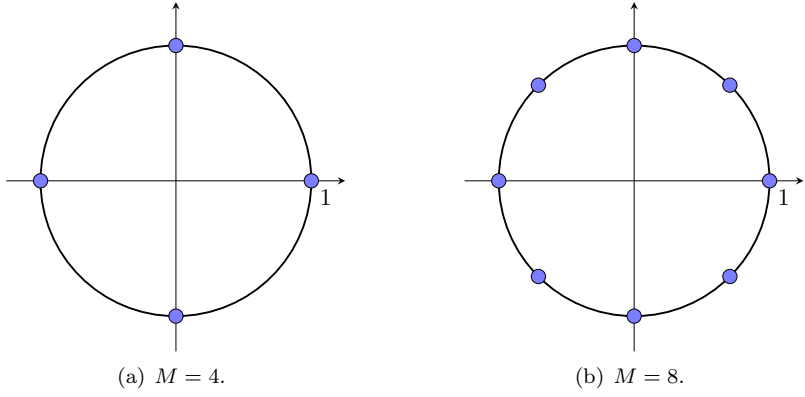

	Let
	\begin{equation}\label{DINGYI}
	\begin{aligned}
	\widehat Q=\begin{bmatrix}
	\RE(Q) & -\IM(Q)\\
	\IM(Q) & \RE(Q)
	\end{bmatrix}=({\hat q}_{jk})_{2n\times 2n},
	\ \hat{c}=\begin{bmatrix}
	\RE(c)\\
	\IM(c)
	\end{bmatrix}\text{, and }
	y=\begin{bmatrix}
	\RE(x)\\
	\IM(x)
	\end{bmatrix}\in\mathbb R^{2n}.
	\end{aligned}
	\end{equation}
	Problem \cref{CQP} can be equivalently written as the following real form:
	\begin{equation}
	\tag{{\rm{RQP}}}
	\label{RQP}
	\begin{aligned}
	& \underset{y\in {\mathbb{R}}^{2n}}{\min}
	& & y^\top\widehat Qy+2\hat{c}^\top y\\
	& \ \text{s.t.}
	& & \left(y_j,\,y_{n+j}\right) \in \mathcal Y,  \ j = 1, \ldots,n.
	\end{aligned}
	\end{equation}
		
	Let $ t=(\bar t_1^\top,\ldots,\,\bar t_n^\top)^\top\in\mathbb{R}^{nM} $ where $\bar t_j\in\mathbb R^M$ is the assignment variable corresponding to $(y_j,\,y_{n+j})$, i.e.,
	\[
	(\bar t_j)_k =\left\{\begin{array}{ll} 1, & \hbox{if }  (y_j,\,y_{n+j})=(\cos\theta_k,\,\sin\theta_k);\\
	0,& \hbox{otherwise.}
	\end{array}\right.
	\]
	By the above definition, the constraints in problem \cref{RQP} can be equivalently written as
	\[\begin{bmatrix}
	y_j\\
	y_{n+j}
	\end{bmatrix}=\sum_{k = 1}^M (\bar t_j)_k\begin{bmatrix}
	\cos\theta_k\\
	\sin\theta_k
	\end{bmatrix},\ \bar t_j\in \{0,\,1\}^M,\ \boldsymbol{e}^\top\bar t_j = 1,\ j = 1, \ldots,n.\]
	Then problem  \cref{RQP} can be equivalently written as
	\begin{equation}
	\label{RQP2}
	\begin{aligned}
	& \underset{y\in\mathbb R^{2n},\atop t\in\mathbb R^{nM}}{\min}
	& & y^\top\widehat Qy+2\hat{c}^\top y\\
	& \ \,\text{s.t.}
	& & y_j=\alpha^\top\bar t_{j},\ j=1, \ldots,n,\\
	&&& y_{n+j}=\beta^\top\bar t_{j},\ j=1, \ldots,n,\\
	&&& \boldsymbol{e}^\top\bar t_{j}=1,\ j=1, \ldots,n,\\
	&&& t\in\{0,\,1\}^{nM},
	\end{aligned}
	\end{equation}
	where
	\begin{equation}\label{def-alphabeta}
	\begin{aligned}
	\alpha&=\left(  \cos 0 ,\,\cos\left(\frac{2\pi}{M}\right) ,\ldots ,\,\cos\left(\frac{2(M-1)\pi}{M}\right) \right) ^\top\in\mathbb{R}^M\text{ and}\\
	\beta&=\left(  \sin 0 ,\,\sin\left(\frac{2\pi}{M}\right) ,\ldots ,\,\sin\left(\frac{2(M-1)\pi}{M}\right) \right) ^\top\in\mathbb{R}^M.
	\end{aligned}
	\end{equation}
	
	We now eliminate the variables $y_j$ for $ j=1,\ldots,2n $, based on the constraints in problem \cref{RQP2}. Let
	\begin{equation}\label{AB}
	\begin{aligned}
	A=I\otimes \alpha^\top\in {\mathbb{R}}^{n\times nM},\ B=I\otimes \beta^\top\in {\mathbb{R}}^{n\times nM}\text{, and }P={\begin{bmatrix}
		A\\B\end{bmatrix}}\in {\mathbb{R}}^{2n\times nM}.
	\end{aligned}
	\end{equation}
	We obtain the following quadratic assignment problem:
	\begin{equation}
	\tag{{\rm{QAP}}}
	\label{QAP}
	\begin{aligned}
	& \underset{t\in {\mathbb{R}}^{nM}}{\min}
	& & h(t)\triangleq t^\top Gt+2w^\top t\\
	& \ \,\text{s.t.}
	& & \boldsymbol{e}^\top\bar t_j=1,\ j=1, \ldots,n,\\
	&&& t\in\{0,\,1\}^{nM}, 
	\end{aligned}
	\end{equation}
	where
	\begin{equation}\label{Gw}
	G=P^\top\widehat QP\in\mathbb{R}^{nM\times nM}\text{ and }w=P^\top\hat{c}\in\mathbb{R}^{nM}.
	\end{equation}
	
	Inspired by the sparse formulation in \cite[(2.10)]{cui2018a}, we define the following SQP problem:
	\begin{equation}
	\tag{{\rm{SQP1}}}
	\label{SQP1}
	\begin{aligned}
	& \underset{t\in {\mathbb{R}}^{nM}}{\min}
	& & h(t)\\
	& \ \,\text{s.t.}
	& & \boldsymbol{e}^\top\bar t_j=1,\ j=1, \ldots,n,\\
	&&& t\geqslant 0,\\
	&&& \|t\|_0\leqslant n,
	\end{aligned}
	\end{equation}
	where $ \|t\|_0\leqslant n $ is the \emph{sparse constraint}, denoting that the sparsity (the number of nonzeros) of $t$ is not greater than $n$.
    We have the following result addressing the connection between problems \cref{QAP} and \cref{SQP1}.
    \begin{proposition}\label{qap-sqp}
    	Problems \cref{QAP} and \cref{SQP1} are equivalent.
    \end{proposition}
    \begin{proof}
    	For each feasible point $ t $ of problem \cref{QAP}, there is $ \|\bar t_j\|_0=1 $, implying that $ \|t\|_0\leqslant n $. Consequently, point $ t $ is feasible for problem \cref{SQP1}. On the other hand, for each feasible point $ t $ of problem \cref{SQP1}, it follows that $ \boldsymbol{e}^\top\bar t_j=1 $ for $ j=1,\ldots,n $, and $ \|t\|_0\leqslant n $, implying that $ \|\bar t_j \|_0 =1 $ for $ j=1,\ldots,n $. Therefore, each entry in $ \bar t_j $ must be either zero or one, i.e., $ t\in\{0,\,1\}^{nM} $. This shows that point $t$ is also feasible for problem \cref{QAP}. Therefore, problems \cref{QAP} and \cref{SQP1} are equivalent.
    \end{proof}
	
	By \cref{qap-sqp}, problem \cref{SQP1} is equivalent to the original problem \cref{P}. Specifically, if $x^*$ is a global minimizer of problem \cref{P}, then $t^*$ obtained by the following
	\begin{equation}\label{XT1}
	\begin{bmatrix}
	\RE(x^*)\\
	\IM(x^*)
	\end{bmatrix}=Pt^*
	=\begin{bmatrix}
	A\\
	B
	\end{bmatrix}t^*
	=\begin{bmatrix}
	At^*\\
	Bt^*
	\end{bmatrix}
	\end{equation}
	is a global minimizer of problem \cref{SQP1}. Conversely, for a global minimizer $ t^* $ of problem \cref{SQP1}, one can get a global minimizer $x^*$ of problem \cref{P} by
	\[x^*=At^*+\mathrm{i}Bt^*.\]
	This reveals that there is a one-to-one correspondence between the global minimizers of problem \cref{SQP1} and those of problem \cref{P}.
	
	Next, we partition the matrix $ G $ (defined in \cref{Gw}) as follows:
	\begin{equation}\label{G}
	\begin{aligned}
	{G}=\begin{bmatrix}
	S_{11}& S_{12} & \cdots & S_{1n}\\
	S_{21} & S_{22} & \cdots & S_{2n}\\
	\vdots & \vdots & \ddots & \vdots\\
	S_{n1} &  S_{n2} & \cdots & S_{nn}
	\end{bmatrix},
	\end{aligned}
	\end{equation}
	where $S_{jk}\in \mathbb{R}^{M\times M}$ for $ j,\ k=1, \ldots,n.$ Define a new matrix $\widetilde G\in \mathbb{R}^{nM\times nM}$ obtained by removing the diagonal blocks in $ G $, i.e., $\widetilde G \triangleq G -\widetilde D$. Here $ \widetilde D $ is the matrix with diagonal blocks $ S_{11},\ldots,\,S_{nn} $, denoted as
	\begin{equation}\label{tildeG}
	\widetilde{D}=\Diag(S_{11},\ldots,\,S_{nn}).
	\end{equation}
	We have the following result.

	\begin{theorem}\label{thm-sqp-sqp2}
		Problem \cref{SQP1} is equivalent to the following problem:
	\begin{equation}
	\tag{{\rm{SQP2}}}
	\label{SQP2}
	\begin{aligned}
	& \underset{t\in {\mathbb{R}}^{nM}}{\min}
	& & f(t)\triangleq t^\top\widetilde{G}t+2w^\top t\\
	& \ \,{\rm s.t.}
	& & \boldsymbol{e}^\top\bar t_j=1,\ j=1, \ldots,n,\\
	&&& t\geqslant 0,\\
	&&& \|  t \|_0\leqslant n .
	\end{aligned}
	\end{equation}
	\end{theorem}
	\begin{proof}
    The proof is relegated to \cref{app-thm-sqp-sqp2}. 
    \end{proof}
    
    Below, we give a property of the objective function $ f(t) $ in problem \cref{SQP2} stating that $ f(t) $ is a linear function with respect to $ \bar{t}_j $, which follows from the fact that the diagonal block in $ \widetilde{G} $ is zero. Such a property is similar to that in \cite[Proposition 3]{cui2018a} for hypergraph matching.
    \begin{proposition}\label{prop-linear}
    	For each block $ \bar t_j,\ j=1, \ldots,n, $ $ f(t) $ in problem \cref{SQP2} is a linear function of $ \bar t_j $, i.e., $ \nabla_{\bar t_j} {f}(t) $ is independent of $ \bar t_j $.
    \end{proposition}
	Due to \cref{prop-linear}, given a particular block $\bar t_j$, the function $f(\cdot)$ can be written as
	\begin{equation}\label{eq-f}
	f(t) = \nabla _{\bar t_j}f(t)^\top \bar t_j + f^{-j}(t_{-j}), \ \ t_{-j} \triangleq (\bar t_1,\cdots, \bar t_{j-1}, \bar t_{j+1},\cdots,\bar t_n)^\top\in\mathbb{R}^{(n-1)M}.
	\end{equation}
	Here  $ \nabla_{\bar t_j} f(t)$ is only related to $t_{-j}$,  and   $ f^{-j}(t_{-j})$ represents the part in $f(\cdot)$ which is only related to $  t_{-j}$.
	 
	\section{Relaxation for MIMO Detection}\label{sec-tightness}
	In this section, we first show the equivalence between problem \cref{SQP2} and its relaxation problem obtained by dropping the sparse constraint. 
	Then we present the properties of the relaxation problem as well as its relations to SDRs.
	
	\subsection{Relaxation of Problem \cref{SQP2}}
	By dropping the sparse constraint in problem \cref{SQP2}, i.e., $ \|t \|_0\leqslant n $, we get the following relaxation problem:
	\begin{equation}
	\tag{{\rm{RSQP}}}
	\label{RSQP}
	\begin{aligned}
	& \underset{t\in {\mathbb{R}}^{nM}}{\min}
	& & f(t)\\
	& \ \text{s.t.}
	& & \boldsymbol{e}^\top\bar t_j=1,\ j=1, \ldots,n,\\
	&&& t\geqslant 0.
	\end{aligned}
	\end{equation}
	The following shows that problem \cref{RSQP} is actually equivalent to problem \cref{SQP2}. 
		\begin{theorem}\label{thm-global}
		There exists a global minimizer $ t^* $ of problem \cref{RSQP} such that $ \|t^*\|_0=n $. As a result, $ t^* $ is a global minimizer of problem \cref{SQP2}. 
	\end{theorem}
	
\begin{proof} 
 We proceed the proof by showing that for problem \cref{RSQP}, there   exists a global optimal solution $ t^* $ such that each block $ \bar t^*_j $ is an extreme point of  the simplex set $\Omega$  defined by $\Omega= \{\nu\in\mathbb{R}^M\mid \nu^\top \boldsymbol{e}=1,\,\nu\geqslant0\}.
$
	Let $ t^\circ $ be a global optimal solution of problem \cref{RSQP}. Suppose that there exists one block as $ \bar t^\circ_{j'} $, such that $ \bar t^\circ_{j'} $ is not an extreme point of $ \Omega $ (i.e., $ \|\bar t^\circ_{j'}\|_0>1 $). Clearly, point $\bar t^\circ_{j'}$ is an optimal solution of the linear programming problem with simplex constraint
	\begin{equation}\label{r-2}
	\underset{\nu\in\Omega}{\min}\,  \nabla _{\bar t_{j'}}f(t^\circ)^\top \nu + f^{-j'}(t^\circ_{-j'}),
	\end{equation}
	where $ \bar t_{-j'}^\circ$ is defined similarly as in \cref{eq-f}. 
	From the basic linear programming theory, there must exist an extreme point $ \nu^1\in\Omega $, such that $ \nu^1 $ is an optimal solution of problem \cref{r-2}. Then we must have
	$
	\nabla _{\bar t_{j'}}f(t^\circ)^\top \nu^1 + f^{-j'}(t^\circ_{-j'})=f( t^\circ).
	$
	Define a new point $ t^1\in\mathbb{R}^{nM} $ by
	$
	\bar t^1_{j'}=\nu^1,\ \bar t^1_j=\bar t^\circ_j,\ j\neq j',\ j=1,\ldots,\,n.
	$
	We have
	$
	f(t^1)=f(t^\circ)
	$
	and, hence $ t^1 $ is a global minimizer of problem \cref{RSQP}. If all of the blocks in $ t^1 $, i.e., $ \bar t^1_j $, $ j=1,\ldots,\,n $, are extreme points of the set $ \Omega $, let $ t^*=t^1 $. The proof is finished. Otherwise, repeat the above process. 
    After at most $ k $ steps ($ k\leqslant n $), one will reach a global minimizer $ t^*\triangleq t^k $, such that all of the blocks in $ t^* $ are extreme points of the set $ \Omega $.   This completes the proof. 
    \end{proof}

\begin{remark}
The proof of \cref{thm-global} makes use of the properties of linear programming. Another way to prove the result is to apply Corollary 2 in \cite{cui2018a} as well as \cref{prop-linear}.
\end{remark}
    \begin{remark}\label{rem-rounding}
	Suppose that one gets a global minimizer of problem \cref{RSQP}, denoted as $t^\circ\in\mathbb R^{nM}$. As pointed out in \cite[Remark 3]{cui2018a}, we can get a global minimizer $ t^\divideontimes $ of problem \cref{SQP2} in the following way. For each block $ \bar t^\divideontimes_j $, pick up any nonzero entry in $ \bar t^\circ_j $, say $ p_j $, and set
	\begin{equation}\label{alg-0}
	\left(\bar t^\divideontimes_j \right)_{p_j}=1\text{, and }\left(\bar t^\divideontimes_j \right)_l=0,\ l\in\{1,\ldots,M\}\backslash p_j.
	\end{equation}
	Notice that $\bar t_j^\divideontimes$ corresponds to an extreme point of the simplex $\Omega$, which is an optimal solution of problem \cref{r-2} with $j'$ replaced by $ j$.
	Repeatedly applying the above rounding procedure, we will obtain a global minimizer of problem \cref{SQP2}.
	\end{remark}

   \begin{remark}\label{remark-new1} From the strong NP-hardness of problem \cref{P} and the equivalence between problems \cref{SQP2,RSQP} (cf. \cref{thm-global}), problem \cref{RSQP} is also strongly NP-hard. However, problem \cref{RSQP} enjoys more advantages than problem \cref{P}. Firstly, problem \cref{RSQP} is a continuous optimization problem so that the local information can be used to design efficient algorithms whereas problem \cref{P} is a discrete optimization problem. Moreover, the feasible region of problem \cref{RSQP} is described by simplex constraints, which are relatively simple. In terms of the objective function, although it is nonconvex, it is a quadratic function. In  particular, we have shown in \cref{prop-linear} that it is a linear function for each block $\bar t_j$ (with all the others being fixed). Consequently, such a special structured nonlinear programming problem with simplex constraints provides us more freedom to explore various numerical algorithms to solve the problem. In other words, by transforming the discrete problem \cref{P} into the continuous optimization problem \cref{RSQP}, we can make full use of various techniques and algorithms in nonlinear optimization.
    \end{remark}

\subsection{Properties of Relaxation Problem \cref{RSQP}}
An interesting question is that under which condition, the relaxation problem \cref{RSQP} admits a unique global minimizer, which corresponds to the vector of transmitted
signals $x^*$ in \cref{MIMO} by \cref{XT1}.  To answer this question,  we first characterize the condition under which, problem (RSQP) admits a unique global minimizer. 	
	\begin{theorem}\label{thm-unique-rsqp2}
		Suppose that $ t^* $ is the unique global optimal solution of problem \cref{SQP2}. Then, $ t^* $ is the unique global minimizer of problem \cref{RSQP}.
	\end{theorem}
	\begin{proof}
		We use the contradiction argument. Assume that $ t^* $ is not the unique global minimizer of problem \cref{RSQP}, then there must exist another global minimizer $ t^\circ $ of problem \cref{RSQP} such that $ t^\circ \neq t^* $. This, together with the assumption that $ t^* $ is  the unique global minimizer of problem \cref{SQP2} and $\boldsymbol{e}^T \bar t^\circ_j = 1 $, implies that $ \|t^\circ \|_0 \geqslant n+1 $ must hold, and hence there must exist a block $\bar t^\circ_j$ such that $\|\bar t^\circ_j\|_0 \geqslant 2 $. Without loss of generality, let $ \|\bar t^\circ_1 \|_0 \geqslant 2 $, $ (\bar t^\circ_1)_1 >0 $, and $ (\bar t^\circ_1)_2 >0 $.  Applying the rounding procedure in \cref{alg-0} by setting
		\[
		(\bar t^\circ_1)_j=\left\{\begin{array}{ll} 1,&\text{if }j=1;\\
		0,&\text{if }j=2,\ldots,M,
		\end{array}\right.\text{ and }(\bar t^\circ_1)_j=\left\{\begin{array}{ll} 1,&\text{if }j=2;\\
		0,&\text{if }j=1,3,\ldots,M,
		\end{array}\right.
		\]
		respectively, we will obtain two different global minimizers of problem \cref{RSQP}. Repeatedly applying the rounding procedure in \cref{alg-0} to other blocks of these two points, we can obtain two \emph{different} global minimizers of problem \cref{SQP2}, which contradicts with the assumption that $ t^* $ is the unique global minimizer of problem \cref{SQP2}.
		Consequently, $ t^* $ is the unique global minimizer of problem \cref{RSQP}.
	\end{proof}

		\Cref{thm-unique-rsqp2} implies that if the vector of transmitted signals $ x^* $ is the unique global minimizer of problem \cref{P}, then the corresponding $ t^* $ obtained via \cref{XT1} is a unique global minimizer of problem \cref{RSQP}. The remaining question is under which condition, $t^*$ is the unique global minimizer of problem \cref{SQP2}. To address this question, we need the definition of tightness and the enhanced SDR in \cite{lu2019tightness}.
 \begin{definition}\label{def-tightness}
		An SDR of problem \cref{P} is called tight if the following two conditions hold: the gap between the SDR and problem \cref{P} is zero; and the SDR recovers the true vector of transmitted signals.
	\end{definition}
	
	The enhanced SDR in \cite{lu2019tightness} is briefly described as follows:
    \begin{equation}
    \tag{{\rm{ERSDR1}}}
    \label{ERSDR1}
    \begin{aligned}
    &\underset{y\in {\mathbb{R}}^{2n},\ t\in {\mathbb{R}}^{nM},\atop Y\in {\mathbb{R}}^{2n\times 2n}}{\min}
    & & \langle\widehat Q,\, Y\rangle+2\hat{c}^\top y \\
    & \qquad \text{s.t.}
    & & \mathbf{Y}(j)=\sum\limits_{k=1}^{M}(\bar t_j)_kU_k,\ j=1,\ldots,n,\\
    &&&  \sum\limits_{k=1}^{M}(\bar t_j)_k=1,\ j=1,\ldots,n,\\
    &&& \begin{bmatrix}
    1 & y^\top\\
    y & Y
    \end{bmatrix}\succeq 0,\\
    &&& t\geqslant 0,
    \end{aligned}
    \end{equation}
	where $\widehat Q$ is defined as in \cref{DINGYI}, 
	$$
	\begin{aligned}
	\mathbf{Y}(j)=\begin{bmatrix}
	1 & y_j & y_{n+j}\\
	y_j & Y_{jj} & Y_{j(n+j)}\\
	y_{n+j} & Y_{(n+j)j} & Y_{(n+j)(n+j)}
	\end{bmatrix},\ j=1,\ldots,n,
	\end{aligned}
	$$
	and
	\[\begin{aligned}
	U_k=\begin{bmatrix}
	1\\ \cos\theta_k\\ \sin\theta_k
	\end{bmatrix}\begin{bmatrix}1& \cos\theta_k& \sin\theta_k\end{bmatrix}, \ k = 1,\ldots,M.
	\end{aligned}
	\]
	We have the following result.
	\begin{theorem}\label{thm-unique-add} Let $t^*\in \mathbb{R}^{nM}$ be the vector corresponding to the vector of transmitted
signals $x^*\in \mathbb{C}^n$ in \cref{MIMO}.
If condition \cref{Cond2} holds, $t^*$ is a unique global minimizer of problem {\cref{SQP2}}.
\end{theorem}
\begin{proof}Note that under condition \cref{Cond2}, problem \cref{ERSDR1} is tight \cite[Theorem 4.4]{lu2019tightness}. By the proof in \cite[Theorem 4.2, Corollary 4.3, Theorem 4.4]{lu2019tightness}, problem \cref{ERSDR1} admits a unique optimal solution, which corresponds to the vector of transmitted
signals $x^*$ in \cref{MIMO}. This, together with the tightness of problem \cref{ERSDR1} and the fact that problem \cref{ERSDR1} is a relaxation of problem \cref{P}, shows that $x^*$ is also a unique solution of problem \cref{P}. Equivalently, under condition \cref{Cond2}, $t^*$ is also a unique global minimizer of problem \cref{SQP2}. 
\end{proof}

\begin{remark}
 \cref{thm-unique-rsqp2,thm-unique-add} imply that under condition \cref{Cond2}, problem \cref{RSQP} is also tight.
\end{remark}

    We illustrate several formulations for the MIMO detection problem in \cref{fig:2a}, which demonstrates the equivalence between problems \cref{P}, \cref{CQP}, \cref{RQP}, \cref{SQP1}, \cref{SQP2}, as well as \cref{RSQP}.  	
    
    \begin{figure}[htbp]
    	\centering
    	\begin{tikzpicture}[line cap=round,line join=round,>=triangle 45,x=1cm,y=1cm]
    	\tikzstyle{every node}=[font=\small,scale=0.85]
    	\draw (-2.48,2.76) node[anchor=north west] {$\cref{P} \Leftrightarrow \cref{CQP} \xLongleftrightarrow[]{Real form} \cref{RQP} \xLongleftrightarrow[sparse\ constraint]{Introducing} \cref{SQP1} \xLongleftrightarrow[]{} \cref{SQP2} \xLongleftrightarrow[]{Relaxation} \cref{RSQP}$};
    	\end{tikzpicture}
    	\centering
    	\caption{The map of equivalent formulations.}
    	\label{fig:2a}
    \end{figure}
    It should be emphasized that problem \cref{RSQP} is a vector based formulation and its size is much smaller (than that of SDRs for problem \cref{P}), and thus it is more suitable to be used for designing algorithms for the large-scale problems. More detailed comparisons between problem \cref{RSQP} and various SDRs will be shown in the next subsection.

\subsection{Relations to the SDRs}\label{sec-extensions}
Recall that the enhanced SDR studied in \cite{lu2019tightness} is \emph{tight} under condition \cref{Cond2}. In fact, we can also show the tightness result of the SDR of our proposed formulation \cref{SQP2} under the same condition. It is easy to check that the following SDR of problem \cref{QAP} proposed in \cite{mobasher2007a}
	\begin{equation}
	\tag{{\rm{ERSDR2}}}
	\label{ERSDR2}
	\begin{aligned}
		& \underset{T\in\mathbb R^{nM\times nM},\atop t\in\mathbb R^{nM}}{\min}
		& & \bar{f}_1(T,\,t)\triangleq \langle G,\, T\rangle+2w^\top t \\
		& \qquad\text{s.t.}
		& & \boldsymbol{e}^\top\bar t_j = 1,\ j=1, \ldots,n,\\
		&&& T_{jj}=\Diag(\bar t_j), \ j=1, \ldots,n,\\
		&&& T\succeq tt^\top,\ t\geqslant 0,
	\end{aligned}
	\end{equation}
	is equivalent to the following SDR of problem \cref{SQP2}
	\begin{equation}
	\tag{{\rm{ERSDR3}}}
	\label{ERSDR3}
	\begin{aligned}
		& \underset{T\in\mathbb R^{nM\times nM},\atop t\in\mathbb R^{nM}}{\min}
		& & \bar{f}_2(T,\,t)\triangleq\langle\widetilde{G},\, T\rangle+2w^\top t \\
		& \qquad \text{s.t.}
		& & \boldsymbol{e}^\top\bar t_j = 1,\ j=1, \ldots,n,\\
		&&& T_{jj}=\Diag(\bar t_j),\ j=1, \ldots,n,\\
		&&& T\succeq tt^\top,\ t\geqslant 0,
	\end{aligned}
	\end{equation}
		where  $ T_{jj}\in \mathbb{R}^{M\times M} $ is the $j$-th diagonal block of $ T $.
With \cite[Theorem 2]{liu2019on}, problem \cref{ERSDR3} is tight for problem \cref{P} under condition \cref{Cond2} for $M\geqslant2$.

	Now, the relations between the series of ``ERSDRs'' and other formulations discussed above can be summarized in \cref{fig:4}. Problems \cref{ERSDR1}, \cref{ERSDR2}, and \cref{ERSDR3} are SDRs of problems \cref{RQP}, \cref{SQP1}, and \cref{SQP2}, respectively. These ``ERSDRs'' are equivalent. 
	\begin{figure}[htbp]
		\centering
		\begin{tikzpicture}[line cap=round,line join=round,>=triangle 45,x=1cm,y=1cm]
		\draw (-1.76,3.68) node[anchor=north west] {\parbox{1.96 cm}{\cref{ERSDR1}     \\ \\  \cref{ERSDR2}  \\    \\  \cref{ERSDR3}}};
		\draw (-1,3.28) node[anchor=north west] {$\Updownarrow$};
		\draw (-1,2.43) node[anchor=north west] {$\Updownarrow$};
		\draw (0.1,3.81) node[anchor=north west] {$\xlongleftarrow[]{\text{SDR}}$};
		\draw (0.1,2.96) node[anchor=north west] {$\xlongleftarrow[]{\text{SDR}}$};
		\draw (0.1,2.11) node[anchor=north west] {$\xlongleftarrow[]{\text{SDR}}$};
		\draw (1.3,4.5) node[anchor=north west] {\parbox{1.8 cm}{\cref{P}  \\  \\  \cref{RQP}  \\   \\  \cref{SQP1}  \\ \\  \cref{SQP2}  \\ \\  \cref{RSQP}}};
		\draw (1.42,4.13) node[anchor=north west] {$\Updownarrow$};
		\draw (1.42,3.28) node[anchor=north west] {$\Updownarrow$};
		\draw (1.42,2.43) node[anchor=north west] {$\Updownarrow$};
		\draw (1.42,1.58) node[anchor=north west] {$\Updownarrow$};
		\end{tikzpicture}
		\centering
		\caption{Relations between the series of ``{\rm ERSDR}s'' and other formulations.}
		\label{fig:4}
	\end{figure}
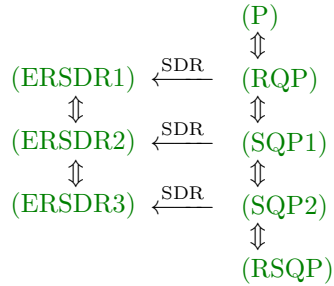

	To conclude this section, we summarize the scale of the above problems in terms of the number of variables and the number of different types of constraints in \cref{tab-1}. In \cref{tab-1}, `$ = $' means equality constraints, `$ \succeq $' means positive semidefinite constraints, and `$ \geqslant $' means lower bound constraints.
	It can be seen from \cref{tab-1} that problem \cref{RSQP} only involves one vector variable $ t\in \mathbb{R}^{nM}_+ $ and $ n $ linear equality constraints, both of which are significantly smaller than those of other relaxation problems. In addition, all constraints in problem \cref{RSQP} are linear. In sharp contrast, all SDR problems contain a positive semidefinite constraint. Our proposed relaxation problem \cref{RSQP} enables us to develop fast algorithms for solving the large-scale MIMO detection problem. Indeed, our proposed PN-QP method for solving the MIMO detection problem is customized based on problem \cref{RSQP}, and as it will be shown in \cref{sec-numerical}, it is much more efficient compared to state-of-the-art ERSDR based approaches.
	\begin{table}[h]
		{\footnotesize
			\caption{Comparison of different relaxations.}\label{tab-1}
			\begin{center}
				\begin{tabular}{|c|c|c|c|c|c|}
					\hline 
					\multirow{2}*{Problem} & \multicolumn{2}{c|}{Number of variables} & \multicolumn{3}{c|}{Number of constraints}  \\ 
					\cline{2-6}
					~ & vector & matrix (size) & $ = $ & $ \succeq $ (size) & $ \geqslant $ \\ 
					\hline 
					\cref{ERSDR1} & $ 2n+nM $ & $ 1(2n\times 2n) $ & $ 6n $ & $ 1(4n\times 4n) $ & $ nM $ \\ 
					\hline
					\cref{ERSDR2} & $ nM $ & $ 1(nM\times nM) $ & $ n+nM(M+1)/2 $ & $ 1(nM\times nM) $ & $ nM $ \\ 
					\hline
					\cref{ERSDR3} & $ nM $ & $ 1(nM\times nM) $ & $ n+nM(M+1)/2 $ & $ 1(nM\times nM) $ & $ nM $ \\ 
					\hline
					\cref{RSQP} & $ nM $ & $ 0 $ & $ n $ & $ 0 $ & $ nM $ \\ 
					\hline 
				\end{tabular}
		\end{center}}
	\end{table}

	\section{Numerical Algorithm for Problem \cref{RSQP}}\label{sec-alg}
	In this section, we present the numerical algorithm for solving problem \cref{RSQP} and discuss its convergence result.
	
	\subsection{Quadratic Penalty Method}
	Recall that problem \cref{RSQP} is a nonlinear programming problem with $n$ simplex constraints. Hence, one can use a solver for constrained optimization problems like \texttt{fmincon} in MATLAB to solve it. However, due to the special property as stated in \cref{rem-rounding}, once the support set of the global minimizer of problem \cref{RSQP} is correctly identified, we can apply the rounding procedure in \cref{alg-0} to obtain a global minimizer of problem \cref{SQP2}. Based on such observations, instead of directly solving problem \cref{RSQP} by treating it as a general constrained optimization problem, we prefer to design an algorithm to (quickly) identify the support set of the global minimizer of problem \cref{RSQP}. Due to this, such an algorithm does not need to strictly satisfy the equality constraints during the algorithmic procedure, i.e., it is reasonable to allow the violations of the equality constraints to some extent. Therefore, we choose the quadratic penalty method to solve problem \cref{RSQP}. 
	More precisely, at each iteration $k$, the quadratic penalty method solves the following subproblem:
	\[
	\begin{aligned}
	& \underset{t\in {\mathbb{R}}^{nM}}{\min}
	& & f_{\omega_k}(t)\triangleq f(t)+\frac{\omega_k}2 \sum_{j = 1}^n(\boldsymbol{e}^\top\bar t_j-1)^2 \\
	& \ \,\text{s.t.}
	& & t\geqslant 0,
	\end{aligned}
	\]
	where $\omega_k>0$ is the penalty parameter. The above subproblem is in general unbounded when $ t\to +\infty $. Therefore, we solve the following subproblem instead
\begin{equation} \label{penaltymethod}
\begin{aligned}
& \underset{t\in {\mathbf{B}}}{\min}
& & f_{\omega_k}(t) 
\end{aligned}
\end{equation}
where $\mathbf{B}\triangleq \{t\in\mathbb{R}^{nM}\mid 0\leqslant t_j\leqslant K,\ j = 1,\ldots, nM\}$, with $ K\geqslant1 $ being a sufficiently large number to guarantee the boundedness of the feasible region. The above problem \cref{penaltymethod} is the penalized subproblem of the following problem
\begin{equation*}\label{ftK}
\begin{aligned}
& \underset{t\in {\mathbb{R}}^{nM}}{\min}
& & f(t) \\
& \ \,\text{s.t.}
& & \boldsymbol{e}^\top\bar t_j=1,\ j=1, \ldots,n, \\
& &&  t\in{\mathbf{B}},
\end{aligned}
\end{equation*}
which is equivalent to problem \cref{RSQP}.
Next, we provide more details on the stopping criteria of the quadratic penalty algorithm and the algorithm for solving the subproblem \cref{penaltymethod}. 
	
	Let $t^k$ be an approximate solution of subproblem \cref{penaltymethod}. 
	As for the stopping criteria, we check whether the support set of $t^k$ is the same as that of the previous step and whether the size of the support set of $\bar t^k_j$ is equal to one for all $j=1,\ldots,n$, i.e.,
	\[\mathcal{K}(t^k)=\mathcal{K}(t^{k-1})\text{ and }\|\bar t^k_j\|_0=1,\ j=1,\ldots,n,\]
	where $ \mathcal{K}(t) $ is the support set of $ t\in\mathbb{R}^{nM} $ defined as
	$\mathcal{K}(t)=\{\ell \mid t_\ell>0,\ \ell=1,\ldots,nM\}.$
	If the above conditions are satisfied, it implies that we reach a feasible point of problem \cref{SQP2} with sparsity $n$, we terminate the iteration. 
	From the numerical point of view, the condition $ \|\bar t^k_j\|_0=1 $ is implemented by
	\begin{equation}\label{eq-stopcriteria-eps}
	\|\bar t^k_j\|_{0,\,\varepsilon}=1,\ j=1,\ldots,n,
	\end{equation}
	where $ \|\bar t^k_j\|_{0,\,\varepsilon} $ denotes the number of elements which are significantly larger than zero, that is, $ \left(\bar t^k_j \right)_l>\varepsilon $, and $ 0<\varepsilon<1 $ is a prescribed small number.

	As for subproblem \cref{penaltymethod}, one equivalent characterization of the stationary point is 
 	\[t^k-\Pi_{\mathbf{B}}(t^k-\nabla f_{\omega_k}(t^k))=0,\]
	where $\Pi_{\mathbf{B}}(t)$ denotes the projection of $t\in\mathbb{R}^{nM}$ onto the set $\mathbf{B}$.
	Here we solve subproblem \cref{penaltymethod} inexactly to get a solution $ t^k $, that is,  $t^k$ satisfies
 	\begin{equation}\label{eq-alg-con2}
\|t^k-\Pi_{\mathbf{B}}(t^k-\nabla f_{\omega_k}(t^k))\|_2\leqslant\tau_k, 	\end{equation}
 	where $ \tau_k\downarrow 0 $.

    Note that subproblem \cref{penaltymethod} is a non-convex quadratic programming problem with simple lower and upper bound constraints. 
    As mentioned above, we prefer to identify the support set of the global minimizer of subproblem \cref{penaltymethod} rather than find the global minimizer itself (in order to reduce the computational cost). The strategy of identifying the active set is therefore crucial in solving subproblem \cref{penaltymethod}. From this point of view, the active set methods are particularly suitable to solve subproblem \cref{penaltymethod}. Therefore, we choose the typical active set method, the projected Newton method proposed in \cite{bertsekas1982projected}, which is demonstrated to be highly efficient in solving large-scale problems such as calibrating least squares covariance matrices \cite{li2011a}.
    
    Overall, we give the details of the PN-QP method in \cref{alg-qp}.
    \begin{algorithm}
    	\caption{PN-QP Method.}
    	\label{alg-qp}
    	\begin{algorithmic}[1]
    		\STATE{Initialization: $ t^0\in\mathbb{R}^{nM} $, $ k:=1 $, $ \rho>1 $, $ \tau_k\downarrow 0 $, $ \varepsilon>0 $, and $ \omega_k>0 $;}	
    		\WHILE{$ k\leqslant{\rm maxiter} $}
    		\STATE{Solve subproblem \cref{penaltymethod} by the projected Newton method to get $ t^k $ such that $ t^k $ satisfies \cref{eq-alg-con2};}
    		\IF{conditions $ \mathcal{K}(t^k)=\mathcal{K}(t^{k-1}) $ and \cref{eq-stopcriteria-eps} are satisfied}
    		\STATE{Break;}
    		\ENDIF
    		\STATE{Set $ \omega_{k+1}:=\rho\omega_k $, $ k:=k+1 $;}
    		\ENDWHILE
    		\RETURN $t^k$.
    	\end{algorithmic}
    \end{algorithm}
 \begin{remark}\label{remark-new2}  Here we would like to highlight that due to the special strategy in identifying the active set for lower and upper constraints, the projected Newton method \cite{bertsekas1982projected} is guaranteed to identify the active set of the stationary point of the problem in the form of subproblem \cref{penaltymethod} \cite[Proposition 2]{bertsekas1982projected}. Moreover, since the second-order information is employed in the projected Newton method, under reasonable assumptions, it is able to converge to a local optimal solution of subproblem \cref{penaltymethod} \cite[Propositions 3, 4]{bertsekas1982projected}.
    \end{remark}

	We have the following classic convergence result of the quadratic penalty method \cite[Chapter 17]{nocedal2006numerical}. Due to the limitation on the length of the paper, we omit the proof here.
	\begin{theorem}\label{thm-PM}
		Suppose that in {\rm \cref{alg-qp}}, the sequence $\{ t^k\} $ satisfies \cref{eq-alg-con2}, $ \tau_k\downarrow 0 $, and $ \omega_k\uparrow +\infty $. Then any accumulation point of the sequence generated by {\rm \cref{alg-qp}} is a stationary point of problem \cref{RSQP}.
	\end{theorem}
	For \cref{alg-qp}, it is possible that the sparsity of the resulting stationary point of $\{t^k\}$ may be greater than $n$. Below we design a special rounding algorithm, which is guaranteed to return  a feasible point of problem \cref{SQP2} with sparsity $n$.
	
	\subsection{Rounding Algorithm} 
	To present the rounding algorithm, we need the following equivalent characterization of stationary points \cite[Lemma 1 (i)]{cui2018a}.  For $ \nu\in\mathbb{R}^{M} $, let	$ \Gamma(\nu)=\{l\mid \nu_l>0\} $ and $ \mathcal{I}(\nu)=\{l\mid \nu_l=0\} $.  	
		\begin{proposition}\label{prop-1} A vector $ t\in\mathbb{R}^{nM} $ is a stationary point of problem \cref{RSQP} if and only if the following conditions hold at $ t $ for $ j=1,\ldots,\,n $,
	\begin{eqnarray}\label{KKT-3}
(\nabla_{\bar t_j}f(t))_s&=&(\nabla_{\bar t_j}f(t))_k,\ \forall\ s,\,k\in \Gamma(\bar t_j),\\\label{KKT-4}
(\nabla_{\bar t_j}f(t))_l&\geqslant& (\nabla_{\bar t_j}f(t))_s,\ \forall\ l\in \mathcal{I}(\bar t_j),\ \forall\ s\in \Gamma(\bar t_j).
\end{eqnarray}
  
\end{proposition}

	The following rounding algorithm for computing a stationary point of problem \cref{RSQP} is slightly different from \cref{alg-0}.  In particular, instead of picking any index of the nonzero entries, we shall pick an index whose corresponding gradient is the smallest, which will help in obtaining a smaller function value. 

\begin{algorithm}
	\caption{Rounding Algorithm}
	\label{alg-round-j}
	\begin{algorithmic}[1]
		\STATE{Initialization: a stationary point $ t^0 $ of problem \cref{RSQP}, $ j=1 $;}
		\WHILE{$ j\leqslant n $}
		\STATE{Let $ s_j\in\arg \underset{l\in\{1,\ldots,\,M\}}{\min}(\nabla_{\bar t_j}f(t^{j-1}))_l$;}
		\STATE{Define $ \nu\in\mathbb{R}^{M} $ by
		$\nu_{s_j}=1,\ \nu_l=0, \ l\neq s_j,\ l\in\{1,\ldots,\,M\}$;}
		\STATE{Define $ t^{j}\in\mathbb{R}^{nM} $ by
			\begin{equation}\label{6}\bar t^{j}_j=\nu,\ \bar t^j_k=\bar t^{j-1}_k,\ k\neq j,\ k\in\{1,\ldots,\,n\}\backslash \{j\};\end{equation}}
		\STATE{Let $ j:=j+1 $;}
		\ENDWHILE
		\RETURN $t^n$.
	\end{algorithmic}
\end{algorithm}
We have the following properties about \cref{alg-round-j}.
\begin{proposition}\label{prop-3}
	Let $ t^0 $ be a stationary point of problem \cref{RSQP}.  Running {\rm \cref{alg-round-j}} with input $t^0$, we have $\|t^n\|_0=n$ and
	\begin{equation}\label{eq-flag}
	f(t^n)\leqslant f(t^0).
	\end{equation}
\end{proposition}
\begin{proof}The sparsity of $t^n$ can be obtained directly by the process of \cref{alg-round-j}. For \cref{eq-flag}, by the definition of $ t^j $ in \cref{6}, \cref{prop-linear}, as well as \cref{prop-1}, one can obtain that $f(t^j)$ does not exceed $f(t^{j-1})$, $j = 1,\dots, n-1$, giving \cref{eq-flag}. \end{proof}

\begin{remark}
\cref{prop-3} reveals that for a stationary point $t^0$ returned by \cref{alg-qp}, the rounding procedure \cref{alg-round-j} will return a feasible point, whose sparsity is $n$ and whose   function value does not exceed $f(t^0)$.
\end{remark}

	\subsection{Exact Detection of PN-QP}
	To further discuss under which condition the PN-QP method has an exact detection guarantee, i.e., the PN-QP method is guaranteed to return the optimal solution $t^*$ corresponding to the vector of transmitted
signals $x^*$, let us denote
\begin{equation}\label{eq-barQ}
\mathring{Q}=H^\dagger H-\Diag(\diag(H^\dagger H))\text{ and }\overline Q = H^\dagger H - \frac12\Diag(\diag(H^\dagger H)).
 \end{equation}
 We need the following two lemmas whose proofs are elementary and therefore were provided in a separate technical report\footnote{http://lsec.cc.ac.cn/$ \sim $yafliu/technical\_report\_MIMO.pdf}.
\begin{lemma}\label{lem-1}
	Let $ x^* $ be the vector of transmitted signals satisfying \cref{MIMO} and $t^*\in\mathbb{R}^{nM}$ be the corresponding vector by \cref{XT1}. Let $ {x} $ be any feasible point of problem \cref{P} with $ {x}\neq x^* $ and (similarly) $ t\in\mathbb{R}^{nM} $ be the corresponding vector, that is,
	\[\begin{bmatrix}
	\RE( {x})\\
	\IM( {x})
	\end{bmatrix}=Pt=\begin{bmatrix}
	At\\
	Bt
	\end{bmatrix},\]
	where $ A $, $ B $, and $ P $ are defined in \cref{AB}.  Furthermore, assume that the phases for $  {x}_j $ and $ x^*_j $ are $ \theta_j $ and $ \theta^*_j $, respectively, $j = 1,\ldots, n$. We have the following results:
	\begin{eqnarray}\label{r2-7}
	&&\|x^*-  x\|_2^2 = 4\sum_{j = 1}^n\sin^2\left(\frac{\theta^*_j-\theta_j}{2}\right), \\
	\label{r2-0}
	&&\|x^*-  x\|_1 = 2\sum_{j = 1}^n\sin\left(\left|\frac{\theta^*_j-\theta_j}{2}\right| \right), \\\label{r2-2}
	&&(t^*-t)^\top \widetilde{G}(t^*-t)=\langle x^*- {x},\,\mathring{Q}(x^*- {x})\rangle, \\\label{r2-3}
	&&(t^*-t)^\top (Gt^*+w)=\langle-H^\dagger v,\,x^*- {x}\rangle, \\\label{r2-20}
	&&2(t^*)^\top\widetilde{D}(t^*-t)=\langle x^*- {x},\,\Diag(\diag(Q)) (x^*-  x)\rangle. 
	\end{eqnarray}
\end{lemma}
 \begin{lemma}\label{lem-2}
$
{ \left|\langle\xi,\,\xi^*\rangle \right| \leqslant \sqrt 2\|\xi\|_\infty\|\xi^*\|_1 , \ \forall\ \xi, \ \xi^*\in\mathbb{C}^n.}
$
\end{lemma}

With the above two lemmas, we have the following result.

\begin{lemma}\label{thm2-10}
	Let $ t,\ t^*,\  {x} $, and $ x^* $ be defined as in {\rm \cref{lem-1}}. 
 If the following condition holds
	\begin{equation}\label{r2-24-two}
	 \sqrt 2\lambda_{\min}(\overline{Q})  \sin \left(\frac{\pi}{M}\right)> \|H^\dagger v\|_\infty,
	\end{equation}
	then there is
	$
	\nabla f(t)^\top(t^*-t)<0.
	$
\end{lemma}
\begin{proof} With $\nabla f(t) = 2\widetilde Gt + 2w$, there is
	\begin{eqnarray}\nonumber
	&&\nabla f(t)^\top(t^*-t)\\\nonumber
	&=&2(\widetilde{G}t+w)^\top(t^*-t)\\\nonumber
	&=&2(\widetilde{G}(t-t^*))^\top(t^*-t)+2(\widetilde{G}t^*+w)^\top(t^*-t)\\\nonumber
	&=&-2(t^*-t)^\top\widetilde{G}(t^*-t)+2(Gt^*+w)^\top(t^*-t)-2(\widetilde{D}t^*)^\top(t^*-t)\ \ ({\rm by\  \cref{tildeG}})\\\nonumber
	&=&-2\langle x^*- {x},\,\mathring{Q}(x^*- {x})\rangle+2\langle-H^\dagger v,\,x^*- {x}\rangle - \langle x^*- {x},\,\Diag(\diag(Q))(x^*- {x})\rangle\ \ ({\rm by\ \cref{lem-1}})\\\nonumber
  &=&-2\langle x^*- {x},\,\overline{Q}(x^*- {x})\rangle+2\langle-H^\dagger v,\,x^*- {x}\rangle\ \ {\rm(by\ \cref{eq-barQ}})\\\nonumber
	&\leqslant&-2\lambda_{\min}(\overline{Q}) \|x^*- {x}\|^2_2+2\sqrt 2\|H^\dagger v\|_\infty\|x^*- {x}\|_1\ \ ({\rm by\ \cref{lem-2}})\\\nonumber
	&=& -8\lambda_{\min}(\overline{Q})\sum_{j = 1}^n \sin^2\left(\frac{\theta_j^*-\theta_j}{2}\right)+ 4\sqrt 2\|H^\dagger v\|_\infty\sum_{j = 1}^n \sin\left(\left|\frac{\theta_j^*-\theta_j}{2} \right|\right)\ \ ({\rm by\ \cref{lem-1}})\\\nonumber
	&=& 4\sqrt 2\sum_{j = 1}^n \left(-\sqrt 2\lambda_{\min}(\overline Q)\sin^2\left(\frac{\theta_j^*-\theta_j}{2}\right)+ \|H^\dagger v\|_\infty\sin\left(\left|\frac{\theta_j^*-\theta_j}{2}\right|\right)\right)
		\\\nonumber
	&\triangleq&4\sqrt 2\sum_{j = 1}^n\Psi\left(\sin\left(\left|\frac{\theta_j^*-\theta_j}{2}\right|\right)\right).
	\end{eqnarray}
	If \cref{r2-24-two} holds, there is
 	$
 	 \lambda_{\min}(\overline{Q})>0
 	$
	and
	\[
	\frac{ \|H^\dagger v\|_\infty}{\sqrt 2\lambda_{\min}(\overline Q)}<\sin\left(\frac{\pi}{M}\right).
	\]
	Consequently, $\Psi(\cdot)$ is decreasing over the interval $\left[\left.\sin\left(\frac\pi M\right),+\infty\right)\right. $, implying that 
	\[
	\Psi\left(\sin\left(\frac\pi M\right)\right)<0.
	\]
	Therefore, with the fact that $x^*\neq \tilde x$ and $x^*, \ \tilde x \in \mathcal X$, we have
	\[
	\begin{aligned}
	\nabla f(t)^\top(t^*-t)\leqslant&4\sqrt 2\sum_{j = 1}^n\Psi\left(\sin\left(\left|\frac{\theta_j^*-\theta_j}{2}\right|\right)\right)\\
	=&4\sqrt 2\sum_{\theta_j^*\neq\theta_j}\Psi\left(\sin\left(\left|\frac{\theta_j^*-\theta_j}{2}\right|\right)\right)\\
	\leqslant&4\sqrt 2\sum_{\theta_j^*\neq\theta_j}\Psi\left(\sin\left(\frac{\pi}{M}\right)\right) <0.
	\end{aligned}
 	\]
 	The proof is completed.
\end{proof}

With \cref{thm-unique-add}, \cref{thm-PM}, and \cref{thm2-10}, we have the following result.
\begin{theorem}
Under conditions \cref{Cond2} and \cref{r2-24-two}, the sequence generated by {\rm \cref{alg-qp}} will converge to the unique global minimizer $t^*$ of problem {\cref{RSQP}}, which corresponds to the vector of transmitted
signals $x^*$ in \cref{MIMO}. 
\end{theorem}
 \begin{proof}
 Under condition \cref{Cond2}, \cref{thm-unique-add} implies that $t^*$ is the unique global minimizer of problem \cref{RSQP}. Together with \cref{thm2-10}, any feasible point of problem \cref{RSQP} with sparsity $n$ other than $t^*$ is not a stationary point, since $ t^*-t $ is a descent direction of the function $ f $ at $ t $ satisfying $ \left(\nabla f(t)\right)^\top(t^*-t)<0 $. Consequently, among all the points with sparsity $n$, $t^*$ is the unique stationary point of problem \cref{RSQP}. By \cref{thm-PM}, the accumulation point of the sequence generated by \cref{alg-qp} will converge to $t^*$, which  corresponds to the vector of transmitted
signals $x^*$. The proof is completed.
 \end{proof}
\begin{remark} As mentioned in \cite{so2010probabilistic},  if $H$ has i.i.d. standard complex Gaussian entries, then $H^\dagger H$ is very close to a diagonal matrix with a very high probability.  Assume that $H^\dagger H$ is a diagonal matrix, there is
	$
	 Q=\Diag(q_{11},\ldots,\,q_{nn})=2\overline Q.
	$
		In this case, condition \cref{r2-24-two} reduces to
	$\lambda_{\min}(H^\dagger H)\sin \left(\frac{\pi}{M}\right)>\sqrt 2\|H^\dagger v\|_\infty,$
	which is  in general stronger than condition {\cref{Cond2}}.
\end{remark}	

	\section{Numerical Results}\label{sec-numerical}
	In this section, we conduct extensive numerical tests to verify the efficiency of the proposed PN-QP algorithm. The algorithm is implemented in MATLAB (R2017a) and all the experiments are preformed on a Lenovo ThinkPad laptop with Intel dual core i5-6200 CPU (2.30 GHZ and 2.40 GHz) and 8 GB of memory running in Windows 10. We generate the instances of problem \cref{P} following the way in \cite{liu2017a,lu2019tightness}, which is detailed as follows:
	\vspace{6pt}
	\begin{itemize}
		\item[Step 1:]	Generate each entry of the channel matrix $ H\in {\mathbb{C}}^{m\times n} $ according to the complex standard Gaussian distribution (with zero mean and unit variance);
		
		\item[Step 2:]	Generate each entry of the noise vector $ v\in {\mathbb{C}}^m $ according to the complex Gaussian distribution with zero mean and variance $ \sigma^2 $;
		
		\item[Step 3:]	Choose $ k_j $ uniformly and randomly from $ \{0,1,\ldots,M-1\} $, and set $ x^*_j=\mathrm{exp}\left(\frac{2\pi k_j \mathrm{i}}{M}\right)  $ for each $ j\in\{1,\ldots,n\} $, where $ x^* $ is the vector of transmitted signals;
		
		\item[Step 4:]	Compute the vector of received signals $ r\in {\mathbb{C}}^m $ as in \cref{MIMO}.
	\end{itemize}
    \vspace{6pt}
    Generally, in practical digital communications, $ M $ is taken as an exponential power of $ 2 $. Therefore, in our following tests, we always choose $ M=2^l $, where $ l $ is a positive integer.
	In our setting, we define the \emph{signal-to-noise ratio} (SNR) as
	\[
	{\rm SNR}=10\log_{10} \left(\frac{\mathbb{E}[\|Hx^*\|_2^2]}{\mathbb{E}[\|v\|_2^2]}\right) =10\log_{10} \left(\dfrac{m\sigma^2_x}{\sigma^2_v}\right) ,
	\]
	where $ \sigma^2_x=\mathbb{E}[\|x^*\|_2^2] $, $ \sigma^2_v=\mathbb{E}[\|v\|_2^2] $, and $ \mathbb{E}[\cdot] $ is the expectation operator. Then according to our ways of generating instances (i.e., $ \sigma^2_x=n $, and $ \sigma^2_v=m\sigma^2 $), we have $ {\rm SNR}=10\log_{10}\left(\dfrac{n}{\sigma^2} \right)  $ in our tests. Generally, the MIMO detection problem is more difficult when the SNR is low and when the numbers of inputs and outputs are equal (i.e., $ m=n $). 

	\subsection{Performance of PN-QP}
	We set $ \rho=3 $, $ \omega_0=10 $, $ \varepsilon=0.01 $ and $ \tau_k=0.01 $ in \cref{alg-qp} and apply \cref{alg-round-j} as the rounding procedure after running \cref{alg-qp}.
	
	First, we demonstrate the efficiency of PN-QP by an example with $ (m,\,n,\,M)=(4,\,4,\,8) $, and $ {\rm SNR}=30 $ dB. The initial point of PN-QP is chosen as $ t^0=t_{\boldsymbol{e}} $, where
	\begin{equation}\label{te}
	t_{\boldsymbol{e}}=\left(\frac{1}{0.2+M}\right) \boldsymbol{e}\in\mathbb{R}^{nM}.
	\end{equation}
	We choose such an initial point $ t_{\boldsymbol{e}} $ since it is a feasible point of subproblem \cref{penaltymethod}, and it approximately satisfies the equality constraints in \cref{RSQP}, i.e., $ \left(\bar{t}_{\boldsymbol{e}} \right)_j^\top \boldsymbol{e}\approx1$,  $j=1,\ldots,n.  $
	We selectively plot the iterates $ \left\lbrace t^k\right\rbrace  $ in \cref{fig:it-t} with $ k=0 $, $ 10 $, and $ 25 $. In \cref{fig:it-t}, the `$ \ast $' denotes the vector $ t^* $ corresponding to the vector of transmitted signals $ x^* $ in \cref{MIMO} and the `$ \circ $' denotes the iterate $ t^k $ generated by PN-QP. It can be seen from \cref{fig:it-t} that as the iteration goes on, $ t^k $ becomes more and more sparse, and eventually, the support set of $ t^k $ at $ k=25 $ coincides with that of the true minimizer of problem \cref{RSQP}. 
	\begin{figure}[htbp]
		\centering
		\subfigure[$ k=0 $ (the initial point).]{
			\begin{minipage}[t]{0.9\linewidth}
				\centering
				\includegraphics[scale=0.7]{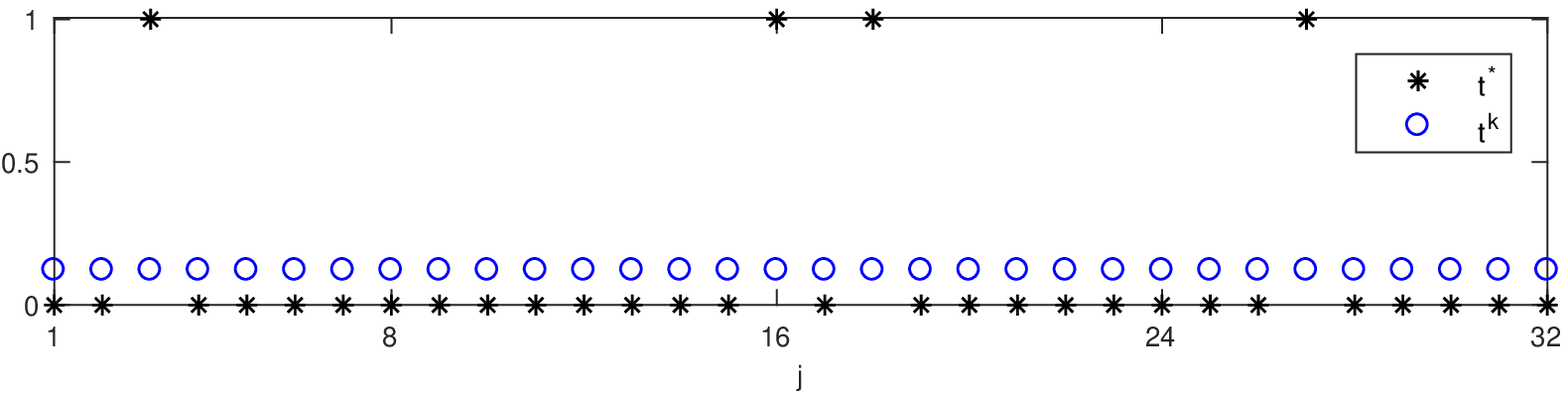}
			\end{minipage}%
		}\\
		\subfigure[$ k=10 $.]{
			\begin{minipage}[t]{0.9\linewidth}
				\centering
				\includegraphics[scale=0.7]{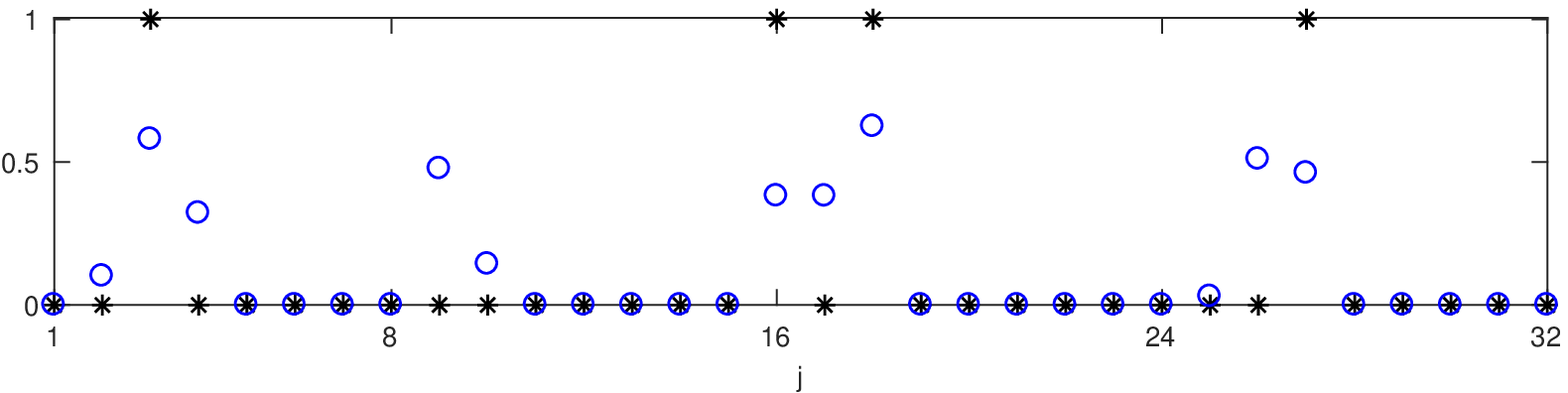}
			\end{minipage}%
		}\\
		\subfigure[$ k=25 $.]{
			\begin{minipage}[t]{0.9\linewidth}
				\centering
				\includegraphics[scale=0.7]{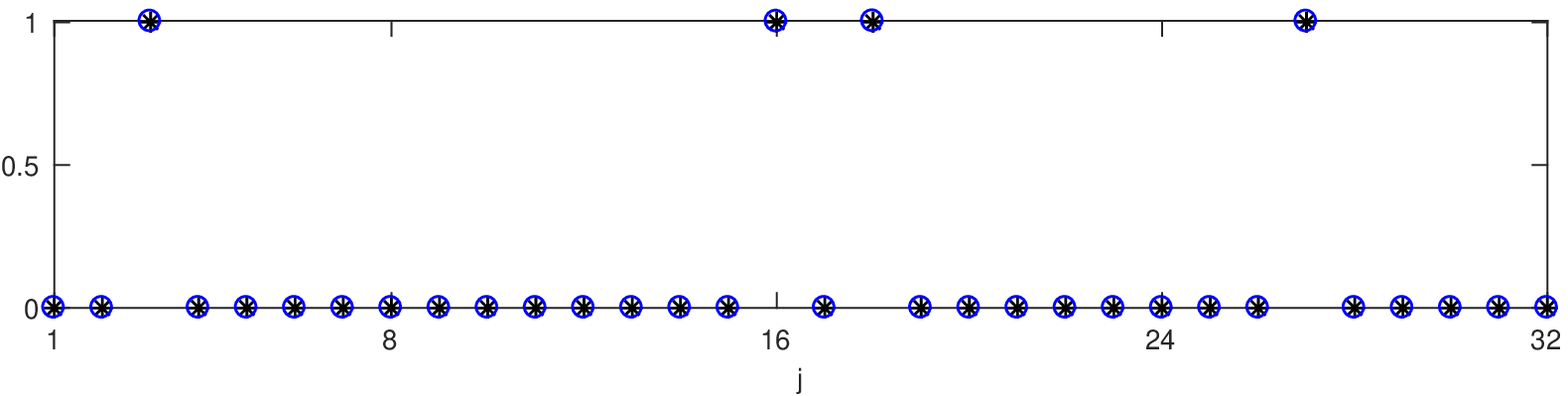}
			\end{minipage}
		}
		\centering
		\caption{Entries in $ t^k $ generated by {\rm PN-QP} with $ k=0,\ 10$, and $ 25 $.}
		\label{fig:it-t}
	\end{figure}
	
	\subsection{Comparison with Other Algorithms}
	In this subsection, we will compare the numerical performance of PN-QP for solving problem \cref{RSQP} with different models and the corresponding algorithms, which are detailed below. 
	\bit
	    \item Problem \cref{P} solved by GPM \cite{liu2017a}: GPM is essentially a gradient projection method whose projection step is taken directly over the discrete set $ \mathcal{X} $. Due to its low computational complexity, it is able to solve the large-scale problem. Moreover, in our implementation, we modify its output by choosing the best point among all generated iterates (to improve its performance), instead of simply using the last iterate $ x^k $ as the output.
	    \item Problem \cref{P} solved by SD\footnote{The code is downloaded from https://ww2.mathworks.cn/matlabcentral/fileexchange/22890-sphere-decoderfor-mimo-systems and modified by adopting the techniques proposed in \cite{chan2002a} to further improve its efficiency.} \cite{o2000lattice}: SD is a typical tree search based method which searches for constellation points limited to a sphere with a predetermined radius centered on the vector of received signals $ r $ to find the \emph{global} solution. However, the complexity of SD is generally exponential, and hence it is impractical to solve the large-scale problem.
	    \item Problem \cref{ERSDR1} solved by MOSEK\footnote{https://www.mosek.com} \cite{mosek}: There are several state-of-the-art solvers for SDRs including MOSEK and SDPNAL+ \cite{sun2020SDPNAL, yang2015SDPNAL, zhao2009a, zhao2010a}. Here we choose MOSEK since it is a classical interior-point algorithm based solver and is faster than SDPNAL+ for solving problem \cref{ERSDR1}. As shown in \cref{tab-1}, the size of problems \cref{ERSDR2} and \cref{ERSDR3} is significantly larger than that of problem \cref{ERSDR1}, and the three problems are (mathematically) equivalent due to \cite[Theorem 1]{liu2019on} and the discussions in \cref{sec-extensions}. Therefore, in our following tests, we do not compare the performance of solving problems \cref{ERSDR2} and \cref{ERSDR3}.
	    For $ (t^k,\,y^k,\,Y^k) $ of problem \cref{ERSDR1} returned by MOSEK, we perform the following rounding procedure to obtain a feasible point: project $ (y^k_j,\,y^k_{n+j}) $ to $ \mathcal{Y} $ in \cref{YJIHE}, and get $ (y^\star_j,\,y^\star_{n+j}) $ by
	    \[
	    (y^\star_j,\,y^\star_{n+j})\in\arg\underset{(u_1,\,u_2)\in\mathcal{Y}}{\min}\, (u_1-y^k_j)^2+(u_2-y^k_{n+j})^2,\ j=1,\ldots,n.
	    \]
	    Then return $ x_j=y^\star_j+\mathrm{i}y^\star_{n+j} $, $ j=1,\ldots,n $.
	\eit

    In our experiments, we set $ M\in\{2,\,8,\,16\}$, $ n\in\{16,\,32,\,64,\,128,\,256,\,512\} $, and $ m=n $ or $ m=2n $. 
    We use the following two metrics to evaluate the performance of different algorithms: the symbol error rate (denoted by SER) \cite{liu2017a,wai2011cheap} as well as the running time in seconds (denoted by Time). More specifically, the SER is used to evaluate the detection error rate of different algorithms, which is calculated by
    \[\frac{\text{The number of incorrectly recovered entries compared to }x^*}{\text{The length of transmitted signals }n}.\]
    Time is used to evaluate the speed of different algorithms, which is particularly important for solving large-scale problems. We limit the maximum running time of each algorithm to be 3600 seconds. That is, we will terminate the algorithm if its running time is over $3600$ seconds and we use ``---'' to denote such case. The reported results below are obtained by averaging over 100 randomly generated instances.
	
	{\bf Initial Points.}
	Since both PN-QP and GPM require the initial points, we first compare the effects of different initial points on the detection results of the two algorithms. We test three initial points: the zero vector $ \textbf{0} $, the vector $ t_{\boldsymbol{e}} $ defined as in \cref{te}, and the approximate solution $ t_{ml} $ obtained by the minimum mean square error (MMSE) detector \cite{liu2017a}. 
	We report the results in \cref{tab-initial points}, where for each setting $ (m,\,n,\,M,\,\text{SNR}) $ and each method, the winner is marked in bold among three initial points.
	
	\begin{sidewaystable}
		\centering
		\caption{Comparison of initial points for {\rm PN-QP} and {\rm GPM}.}
		\scalebox{1}{
			\begin{tabular}{|r|r|rrr|rrr|rrr|rrr|}
				\hline 
				\makecell[c]{\multirow{4}*{$ (m,\,n,\,M) $}} &  & \multicolumn{6}{c|}{Time(s)} & \multicolumn{6}{c|}{SER($ \% $)} \\ 
				\cline{3-14}
				& SNR & \multicolumn{3}{c|}{PN-QP} & \multicolumn{3}{c|}{GPM} & \multicolumn{3}{c|}{PN-QP} & \multicolumn{3}{c|}{GPM} \\
				& (dB) & \multicolumn{3}{c|}{\cref{RSQP}} & \multicolumn{3}{c|}{\cref{P}} & \multicolumn{3}{c|}{\cref{RSQP}} & \multicolumn{3}{c|}{\cref{P}} \\  
				\cline{3-14}
				&  & \makecell[c]{$ \textbf{0} $} & \makecell[c]{$ t_{\boldsymbol{e}} $} & \makecell[c]{$ t_{ml} $} & \makecell[c]{$ \textbf{0} $} & \makecell[c]{$ t_{\boldsymbol{e}} $} & \makecell[c]{$ t_{ml} $} & \makecell[c]{$ \textbf{0} $} & \makecell[c]{$ t_{\boldsymbol{e}} $} & \makecell[c]{$ t_{ml} $} & \makecell[c]{$ \textbf{0} $} & \makecell[c]{$ t_{\boldsymbol{e}} $} & \makecell[c]{$ t_{ml} $} \\ 
				\hline 
				\makecell[c]{\multirow{4}*{$ (64,\,32,\,8) $}}& 24 & 0.034 & 0.034 & \textbf{0.012} & 0.001 & 0.001 & \textbf{0.000} & 0.00 & 0.00 & 0.00 & 12.44 & 12.31 & \textbf{0.00}\\
				& 20 & 0.029 & 0.033 & \textbf{0.011} & 0.001 & 0.001 & \textbf{0.000} & 0.00 & 0.00 & 0.00 & 12.13 & 13.13 & \textbf{0.00}\\
				& 16 & 0.030 & 0.036 & \textbf{0.011} & 0.001 & 0.001 & \textbf{0.000} & 0.00 & 0.00 & 0.00 & 7.06 & 9.16 & \textbf{0.00}\\
				& 12 & 0.035 & 0.038 & \textbf{0.012} & 0.001 & 0.001 & \textbf{0.000} & 0.25 & 0.25 & 0.25 & 5.72 & 5.47 & \textbf{0.91}\\
				\hline 
				\makecell[c]{\multirow{4}*{$ (32,\,32,\,8) $}}& 24 & 0.046 & 0.050 & \textbf{0.014} & 0.001 & 0.001 & \textbf{0.000} & 0.19 & \textbf{0.00} & \textbf{0.00} & 56.22 & 58.28 & \textbf{1.25}\\
				& 20 & 0.045 & 0.048 & \textbf{0.017} & 0.001 & 0.001 & \textbf{0.000} & 1.06 & \textbf{0.25} & \textbf{0.25} & 57.50 & 57.19 & \textbf{2.78}\\
				& 16 & 0.058 & 0.052 & \textbf{0.023} & 0.001 & 0.001 & 0.001 & 3.59 & \textbf{2.00} & 4.25 & 57.63 & 55.75 & \textbf{11.09}\\
				& 12 & 0.063 & 0.063 & \textbf{0.031} & 0.001 & 0.001 & 0.001 & 17.59 & \textbf{16.13} & 20.28 & 52.13 & 52.16 & \textbf{25.44}\\
				\hline 
				\makecell[c]{\multirow{4}*{$ (64,\,32,\,16) $}}& 24 & 0.079 & 0.091 & \textbf{0.027} & 0.001 & 0.001 & \textbf{0.000} & 0.00 & 0.00 & 0.00 & 40.75 & 40.41 & \textbf{0.00}\\
				& 20 & 0.082 & 0.095 & \textbf{0.028} & 0.001 & 0.001 & \textbf{0.000} & 0.00 & 0.00 & 0.00 & 30.06 & 30.72 & \textbf{0.06}\\
				& 16 & 0.096 & 0.115 & \textbf{0.036} & 0.001 & 0.001 & \textbf{0.000} & 2.63 & \textbf{2.31} & 2.78 & 32.06 & 29.94 & \textbf{3.22}\\
				& 12 & 0.114 & 0.123 & \textbf{0.057} & 0.001 & 0.001 & 0.001 & 18.13 & \textbf{17.94} & 18.88 & 28.16 & 29.91 & \textbf{18.84}\\
				\hline 
				\makecell[c]{\multirow{4}*{$ (32,\,32,\,16) $}}& 24 & 0.133 & 0.142 & \textbf{0.095} & 0.002 & 0.001 & \textbf{0.000} & 1.44 & \textbf{0.53} & 5.28 & 67.28 & 66.00 & \textbf{15.91}\\
				& 20 & 0.171 & 0.176 & \textbf{0.133} & 0.002 & 0.001 & \textbf{0.000} & 6.75 & \textbf{5.19} & 13.09 & 68.53 & 69.34 & \textbf{22.03}\\
				& 16 & 0.190 & 0.196 & \textbf{0.172} & 0.002 & \textbf{0.001} & \textbf{0.001} & 26.91 & \textbf{25.41} & 32.41 & 72.69 & 71.69 & \textbf{38.63}\\
				& 12 & 0.198 & 0.203 & \textbf{0.158} & 0.001 & 0.001 & 0.001 & 48.34 & \textbf{46.53} & 48.84 & 68.91 & 69.53 & \textbf{52.56}\\
				\hline 
				\makecell[c]{\multirow{4}*{$ (1024,\,512,\,8) $}}& 24 & 0.746 & 0.653 & \textbf{0.245} & 1.623 & 1.802 & \textbf{0.136} & 0.00 & 0.00 & 0.00 & 14.38 & 15.52 & \textbf{0.00}\\
				& 20 & 0.767 & 0.676 & \textbf{0.240} & 0.891 & 1.175 & \textbf{0.134} & 0.00 & 0.00 & 0.00 & 5.82 & 9.28 & \textbf{0.00}\\
				& 16 & 0.837 & 0.748 & \textbf{0.259} & 0.543 & 0.523 & \textbf{0.145} & 0.00 & 0.00 & 0.00 & 1.86 & 2.28 & \textbf{0.00}\\
				& 12 & 1.153 & 1.053 & \textbf{0.583} & 0.351 & 0.351 & \textbf{0.201} & 0.26 & \textbf{0.25} & 0.27 & 0.42 & 0.40 & \textbf{0.38}\\
				\hline 
				\makecell[c]{\multirow{4}*{$ (512,\,512,\,8) $}}& 24 & 1.235 & 1.108 & \textbf{0.461} & 1.519 & 1.471 & \textbf{0.116} & 0.00 & 0.00 & 0.00 & 57.48 & 57.74 & \textbf{0.00}\\
				& 20 & 1.317 & 1.190 & \textbf{0.723} & 1.525 & 1.492 & \textbf{0.141} & 0.00 & 0.00 & 0.00 & 57.78 & 58.20 & \textbf{0.00}\\
				& 16 & 1.810 & \textbf{1.621} & 1.676 & 1.555 & 1.430 & \textbf{0.923} & 0.10 & 0.10 & 0.10 & 58.14 & 58.38 & \textbf{8.55}\\
				& 12 & \textbf{3.101} & 3.155 & 3.520 & 1.514 & 1.470 & \textbf{0.733} & 17.41 & \textbf{17.00} & 20.69 & 58.51 & 59.03 & \textbf{24.07}\\
				\hline 
				\makecell[c]{\multirow{4}*{$ (1024,\,512,\,16) $}}& 24 & 1.469 & 1.431 & \textbf{0.512} & 3.646 & 3.900 & \textbf{0.136} & 0.00 & 0.00 & 0.00 & 55.38 & 55.37 & \textbf{0.00}\\
				& 20 & 1.714 & 1.664 & \textbf{0.651} & 3.815 & 3.948 & \textbf{0.168} & 0.01 & 0.01 & 0.01 & 53.86 & 54.08 & \textbf{0.01}\\
				& 16 & 2.616 & 2.492 & \textbf{1.852} & 4.094 & 3.983 & \textbf{0.295} & \textbf{1.92} & 1.95 & 2.11 & 47.75 & 48.16 & \textbf{2.04}\\
				& 12 & 4.082 & \textbf{3.896} & 4.412 & 1.394 & 1.430 & \textbf{0.407} & \textbf{18.64} & \textbf{18.64} & 19.69 & 25.36 & 25.53 & \textbf{19.82}\\
				\hline 
				\makecell[c]{\multirow{4}*{$ (512,\,512,\,16) $}}& 24 & 3.403 & \textbf{3.116} & 3.193 & 2.311 & 2.401 & \textbf{2.243} & 0.00 & 0.00 & 0.00 & 77.61 & 77.68 & \textbf{23.54}\\
				& 20 & 7.071 & \textbf{6.434} & 8.134 & 2.426 & 2.450 & \textbf{2.235} & \textbf{3.13} & \textbf{3.13} & 8.98 & 77.97 & 77.97 & \textbf{33.35}\\
				& 16 & 6.525 & \textbf{6.355} & 7.881 & 2.337 & 2.586 & \textbf{2.270} & 29.20 & \textbf{29.18} & 32.58 & 77.72 & 77.69 & \textbf{45.58}\\
				& 12 & 6.900 & \textbf{6.686} & 8.725 & 2.326 & \textbf{2.199} & 2.360 & 48.27 & \textbf{48.14} & 50.17 & 78.29 & 78.13 & \textbf{55.01}\\
				\hline 
			\end{tabular}
		}
		\label{tab-initial points}
	\end{sidewaystable}

	It can be observed from \cref{tab-initial points} that in terms of the time for PN-QP, for small scale of $ (m,\,n) $, $ t_{ml} $ takes the smallest time whereas for large scale of $ (m,\,n) $, $ t_{\boldsymbol{e}} $ and $ t_{ml} $ are comparable. Coming to the SER for PN-QP, $ t_{\boldsymbol{e}} $ is more favorable among the three choices. 
	In comparison, the SER by GPM varies quite a lot among the three choices of the initial point, and $ t_{ml} $ is definitely the best initial point for GPM which will lead to a much smaller SER. In terms of the running time, it seems that $ t_{ml} $ for PN-QP is preferable for small-scale problems whereas $ t_{\boldsymbol{e}} $ leads to the smallest running time for PN-QP among the three choices when the size of problem $ n $ is large (i.e., $ n=512 $). For GPM, $ t_{ml} $ is also the winner from the perspective of running time.
	Based on the above observations, in our following test, we choose $ t_{\boldsymbol{e}} $ as the initial point for PN-QP and the MMSE estimator $ t_{ml} $ as the initial point for GPM. Here we would like to highlight that the numerical results with $ n\geqslant 128 $ which we show in this paper have not appeared in literature.
	
	{\bf Results on Problems with $M\leqslant 8$.}
	Next, we compare the performance of the four algorithms in the cases that $ M=2 $ and $ M=8 $.
	We also report the \emph{no interference lower bound} (LB) results. This approach solves the MIMO detection problem with respect to each component $x_j$ assuming all the others being fixed to be the true transmitted signals. Again, the SER is obtained by dividing the total number of incorrectly estimated elements over the length of transmitted signals. The solution returned by LB can be viewed as the best possible result that the MIMO detection problem can be solved theoretically. Therefore, the above no interference LB can be used as the theoretical (and generally unachievable especially in the low SNR scenarios) lower bound of the SER of all the other approaches. 
	
	It can be observed from \cref{tab-all2} and \cref{tab-all} that as the SNR decreases, the SER achieved by each algorithm increases, implying that the problem becomes more difficult. As shown in \cref{tab-all2}, when $ M=2 $, all methods perform well in terms of the SER. 
	
	From \cref{tab-all2} and \cref{tab-all}, one can see that the running time for SD becomes longer and even prohibitively high as the SNR decreases/$ n $ increases, despite that SD provides the best SER among the four algorithms. For the other three algorithms, in general PN-QP provides the best SER, and MOSEK performs better than GPM in terms of the SER. For example, for $ (m,\,n,\,M)=(128,\,128,\,8) $ and $ {\rm SNR}=14 $ dB, PN-QP takes $ 0.277 $ seconds to return a solution with $ {\rm SER}=4.57\% $, whereas MOSEK takes about one second to return a solution with a larger SER $ 7.14\% $, and GPM returns a solution with the SER being $ 17.13\% $ instantly (i.e., $ 0.012 $ seconds). For such example, SD fails to return a solution within one hour. For the running time, GPM for solving problem \cref{P} is the fastest one since it only involves gradient calculation and projection onto a discrete set. PN-QP for solving problem \cref{RSQP} is fairly fast and its running time increases slowly as $ n $ increases. This is due to the vector formulation of problem \cref{RSQP}. MOSEK for solving problem \cref{ERSDR1} is not as fast as GPM and PN-QP. As $ n $ increases, the running time increases much faster than that for PN-QP. Comparing $(m,\,n,\,M)=(256,\,256,\,8) $ with $(m,\,n,\,M)=(512,\,512,\,8)$, it can be observed that as $n$ increases from $256$ to $512$, the running time for MOSEK increases from about $10$ seconds to about $85$ seconds.
	This can be explained by the numbers of the variables (one $2n \times 2n$ matrix variable and one $2n+nM$ vector variable) and constraints (one  $4n\times 4n$ positive semidefinite constraint, $6n$ equality constraints, and $nM$ inequality constraints) in \cref{tab-1}. Therefore, in our subsequent test, we will not include SD and MOSEK.
	
	To better understand the detection performance of the four algorithms, we plot \cref{fig:ser-small}, showing the SER with respect to the SNR for each algorithm. 
	It can be seen from \cref{fig:ser-small} that for $ (m,\,n,\,M)=(16,\,16,\,8) $, SD performs the best since its SER curve coincides with the LB when the SNR is large. MOSEK also performs very well since the curve of MOSEK becomes parallel to the LB, i.e., a constant SER gap. However, this is not the case for PN-QP and GPM. For $ (m,\,n,\,M)=(32,\,16,\,8) $, PN-QP, MOSEK, and SD are competitive, whereas SD is the best one.
	\begin{table}[!htb]
		\centering
		\caption{Time and SER comparison of {\rm PN-QP}, {\rm MOSEK}, {\rm GPM}, and {\rm SD} for solving MIMO detection problems with $M=2$.}
		\scalebox{0.78}{
			\begin{tabular}{|r|r|r|r|r|r|r|r|r|r|r|}
				\hline 
				\makecell[c]{\multirow{3}*{$ (m,\,n,\,M) $}} & \makecell[c]{\multirow{3}*{SNR}} & \multicolumn{4}{c|}{Time(s)} & \multicolumn{5}{c|}{SER($ \% $)} \\ 
				\cline{3-11}
				&  & \makecell[c]{\scriptsize{PN-QP}} & \makecell[c]{\scriptsize{MOSEK}} & \makecell[c]{\scriptsize{GPM}} & \makecell[c]{\scriptsize{SD}} & \makecell[c]{\scriptsize{PN-QP}} & \makecell[c]{\scriptsize{MOSEK}} & \makecell[c]{\scriptsize{GPM}} & \makecell[c]{\scriptsize{SD}} & \makecell[c]{\scriptsize{LB}} \\
				& (dB) & \makecell[c]{\scriptsize{\cref{RSQP}}} & \makecell[c]{\scriptsize{\cref{ERSDR1}}} & \makecell[c]{\scriptsize{\cref{P}}} & \makecell[c]{\scriptsize{\cref{P}}} & \makecell[c]{\scriptsize{\cref{RSQP}}} & \makecell[c]{\scriptsize{\cref{ERSDR1}}} & \makecell[c]{\scriptsize{\cref{P}}} & \makecell[c]{\scriptsize{\cref{P}}} &   \\ 
				\hline 
				\makecell[c]{\multirow{6}*{$ (32,\,32,\,2) $}} & 22 & 0.004 & 0.338 & 0.001 & 0.003 & 0.00 & 0.00 & 0.00 & 0.00 & 0.00\\
				& 20 & 0.004 & 0.329 & 0.001 & 0.003 & 0.00 & 0.00 & 0.00 & 0.00 & 0.00\\
				& 18 & 0.004 & 0.335 & 0.001 & 0.005 & 0.00 & 0.00 & 0.00 & 0.00 & 0.00\\
				& 16 & 0.005 & 0.338 & 0.001 & 0.012 & 0.00 & 0.00 & 0.00 & 0.00 & 0.00\\
				& 14 & 0.004 & 0.348 & 0.001 & 0.047 & 0.00 & 0.00 & 0.03 & 0.00 & 0.00\\
				& 12 & 0.004 & 0.369 & 0.001 & 0.110 & 0.00 & 0.00 & 0.16 & 0.00 & 0.00\\
				\hline
				\makecell[c]{\multirow{6}*{$ (64,\,64,\,2) $}} & 22 & 0.008 & 0.456 & 0.001 & 0.007 & 0.00 & 0.00 & 0.00 & 0.00 & 0.00\\
				& 20 & 0.008 & 0.460 & 0.001 & 0.013 & 0.00 & 0.00 & 0.00 & 0.00 & 0.00\\
				& 18 & 0.008 & 0.469 & 0.001 & 0.035 & 0.00 & 0.00 & 0.00 & 0.00 & 0.00\\
				& 16 & 0.008 & 0.503 & 0.001 & 0.158 & 0.00 & 0.00 & 0.00 & 0.00 & 0.00\\
				& 14 & 0.008 & 0.534 & 0.001 & 1.392 & 0.00 & 0.00 & 0.00 & 0.00 & 0.00\\
				& 12 & 0.008 & 0.607 & 0.001 & 10.251 & 0.00 & 0.00 & 0.09 & 0.00 & 0.00\\
				\hline
				\makecell[c]{\multirow{6}*{$ (128,\,128,\,2) $}} & 22 & 0.025 & 1.256 & 0.003 & 0.815 & 0.00 & 0.00 & 0.00 & 0.00 & 0.00\\
				& 20 & 0.025 & 1.319 & 0.003 & 12.260 & 0.00 & 0.00 & 0.00 & 0.00 & 0.00\\
				& 18 & 0.026 & 1.500 & 0.003 & 284.234 & 0.00 & 0.00 & 0.00 & 0.00 & 0.00\\
				& 16 & 0.030 & 1.723 & 0.004 & 2348.488 & 0.00 & 0.00 & 0.00 & 0.00 & 0.00\\
				& 14 & 0.030 & 1.783 & 0.004 & --- & 0.00 & 0.00 & 0.00 & --- & 0.00\\
				& 12 & 0.027 & 2.158 & 0.003 & --- & 0.00 & 0.00 & 0.02 & --- & 0.00\\
				\hline
				\makecell[c]{\multirow{6}*{$ (256,\,256,\,2) $}} & 22 & 0.111 & 7.249 & 0.013 & --- & 0.00 & 0.00 & 0.00 & --- & 0.00\\
				& 20 & 0.112 & 7.408 & 0.013 & --- & 0.00 & 0.00 & 0.00 & --- & 0.00\\
				& 18 & 0.113 & 7.683 & 0.014 & --- & 0.00 & 0.00 & 0.00 & --- & 0.00\\
				& 16 & 0.115 & 8.502 & 0.014 & --- & 0.00 & 0.00 & 0.00 & --- & 0.00\\
				& 14 & 0.114 & 12.064 & 0.014 & --- & 0.00 & 0.00 & 0.00 & --- & 0.00\\
				& 12 & 0.117 & 15.292 & 0.015 & --- & 0.00 & 0.00 & 0.01 & --- & 0.00\\
				\hline
				\makecell[c]{\multirow{6}*{$ (512,\,512,\,2) $}} & 22 & 0.198 & 58.730 & 0.092 & --- & 0.00 & 0.00 & 0.00 & --- & 0.00\\
				& 20 & 0.200 & 61.319 & 0.092 & --- & 0.00 & 0.00 & 0.00 & --- & 0.00\\
				& 18 & 0.201 & 64.035 & 0.092 & --- & 0.00 & 0.00 & 0.00 & --- & 0.00\\
				& 16 & 0.207 & 71.433 & 0.097 & --- & 0.00 & 0.00 & 0.00 & --- & 0.00\\
				& 14 & 0.216 & 113.457 & 0.106 & --- & 0.00 & 0.00 & 0.00 & --- & 0.00\\
				& 12 & 0.230 & 135.229 & 0.110 & --- & 0.00 & 0.00 & 0.00 & --- & 0.00\\
				\hline 
		\end{tabular}}
		\label{tab-all2}
	\end{table}
	\begin{table}[!htb]
		\centering
		\caption{Time and SER comparison of {\rm PN-QP}, {\rm MOSEK}, {\rm GPM}, and {\rm SD} for solving MIMO detection problems with $M=8$.}
		\scalebox{0.78}{
			\begin{tabular}{|r|r|r|r|r|r|r|r|r|r|r|}
				\hline 
				\makecell[c]{\multirow{3}*{$ (m,\,n,\,M) $}} & \makecell[c]{\multirow{3}*{SNR}} & \multicolumn{4}{c|}{Time(s)} & \multicolumn{5}{c|}{SER($ \% $)} \\ 
				\cline{3-11}
				&  & \makecell[c]{\scriptsize{PN-QP}} & \makecell[c]{\scriptsize{MOSEK}} & \makecell[c]{\scriptsize{GPM}} & \makecell[c]{\scriptsize{SD}} & \makecell[c]{\scriptsize{PN-QP}} & \makecell[c]{\scriptsize{MOSEK}} & \makecell[c]{\scriptsize{GPM}} & \makecell[c]{\scriptsize{SD}} & \makecell[c]{\scriptsize{LB}} \\
				& (dB) & \makecell[c]{\scriptsize{\cref{RSQP}}} & \makecell[c]{\scriptsize{\cref{ERSDR1}}} & \makecell[c]{\scriptsize{\cref{P}}} & \makecell[c]{\scriptsize{\cref{P}}} & \makecell[c]{\scriptsize{\cref{RSQP}}} & \makecell[c]{\scriptsize{\cref{ERSDR1}}} & \makecell[c]{\scriptsize{\cref{P}}} & \makecell[c]{\scriptsize{\cref{P}}} &   \\ 
				\hline
				\makecell[c]{\multirow{6}*{$ (32,\,32,\,8) $}} & 22 & 0.044 & 0.362 & 0.001 & 0.208 & 0.00 & 0.00 & 1.31 & 0.00 & 0.00\\
				& 20 & 0.050 & 0.307 & 0.001 & 1.666 & 0.00 & 0.13 & 2.13 & 0.00 & 0.00\\
				& 18 & 0.051 & 0.340 & 0.001 & 19.923 & 0.00 & 0.25 & 4.06 & 0.00 & 0.00\\
				& 16 & 0.051 & 0.335 & 0.001 & 305.642 & 1.63 & 1.94 & 9.59 & 0.09 & 0.03\\
				& 14 & 0.060 & 0.403 & 0.002 & 2495.598 & 3.13 & 6.88 & 21.88 & 3.13 & 0.00\\
				& 12 & 0.067 & 0.351 & 0.001 & --- & 13.94 & 14.81 & 23.47 & --- & 3.69\\
				\hline
				\makecell[c]{\multirow{6}*{$ (64,\,64,\,8) $}} & 22 & 0.126 & 0.513 & 0.002 & 769.366 & 0.00 & 0.00 & 0.22 & 0.00 & 0.00\\
				& 20 & 0.134 & 0.537 & 0.002 & --- & 0.00 & 0.00 & 0.06 & --- & 0.00\\
				& 18 & 0.132 & 0.507 & 0.003 & --- & 0.00 & 0.39 & 3.09 & --- & 0.00\\
				& 16 & 0.137 & 0.519 & 0.002 & --- & 0.78 & 1.84 & 7.94 & --- & 0.02\\
				& 14 & 0.150 & 0.510 & 0.003 & --- & 5.92 & 7.28 & 15.53 & --- & 0.66\\
				& 12 & 0.188 & 0.506 & 0.002 & --- & 16.34 & 16.44 & 24.97 & --- & 3.38\\
				\hline
				\makecell[c]{\multirow{6}*{$ (128,\,128,\,8) $}} & 22 & 0.156 & 1.763 & 0.004 & --- & 0.00 & 0.00 & 0.11 & --- & 0.00\\
				& 20 & 0.164 & 1.730 & 0.005 & --- & 0.00 & 0.00 & 0.95 & --- & 0.00\\
				& 18 & 0.172 & 1.701 & 0.006 & --- & 0.01 & 0.24 & 2.64 & --- & 0.01\\
				& 16 & 0.189 & 1.675 & 0.010 & --- & 0.38 & 1.88 & 6.73 & --- & 0.06\\
				& 14 & 0.277 & 1.651 & 0.012 & --- & 4.57 & 7.14 & 17.13 & --- & 0.63\\
				& 12 & 0.307 & 1.555 & 0.008 & --- & 15.89 & 15.47 & 23.63 & --- & 3.27\\
				\hline
				\makecell[c]{\multirow{6}*{$ (256,\,256,\,8) $}} & 22 & 0.342 & 10.558 & 0.018 & --- & 0.00 & 0.00 & 0.00 & --- & 0.00\\
				& 20 & 0.353 & 10.290 & 0.029 & --- & 0.00 & 0.02 & 0.98 & --- & 0.00\\
				& 18 & 0.379 & 10.153 & 0.041 & --- & 0.00 & 0.32 & 2.02 & --- & 0.00\\
				& 16 & 0.450 & 10.144 & 0.076 & --- & 0.08 & 2.08 & 7.54 & --- & 0.05\\
				& 14 & 0.587 & 9.799 & 0.097 & --- & 4.11 & 7.45 & 15.85 & --- & 0.75\\
				& 12 & 0.660 & 9.748 & 0.067 & --- & 16.60 & 16.25 & 23.22 & --- & 3.18\\
				\hline
				\makecell[c]{\multirow{6}*{$ (512,\,512,\,8) $}}& 22 & 1.193 & 88.809 & 0.134 & --- & 0.00 & 0.00 & 0.00 & --- & 0.00\\
				& 20 & 1.303 & 91.365 & 0.177 & --- & 0.00 & 0.01 & 0.09 & --- & 0.00\\
				& 18 & 1.436 & 90.253 & 0.297 & --- & 0.00 & 0.25 & 0.88 & --- & 0.00\\
				& 16 & 1.825 & 88.638 & 0.929 & --- & 0.08 & 2.23 & 6.68 & --- & 0.06\\
				& 14 & 3.177 & 86.758 & 1.455 & --- & 2.46 & 7.29 & 16.82 & --- & 0.66\\
				& 12 & 3.482 & 85.408 & 0.765 & --- & 16.41 & 16.15 & 22.97 & --- & 3.17\\
				\hline 
		\end{tabular}}
		\label{tab-all}
	\end{table}
    \begin{figure}[htbp]
    	\centering
    	\subfigure[$ (m,\,n,\,M)=(16,\,16,\,8) $.]{
    		\begin{minipage}[t]{0.48\linewidth}
    			\centering
    			\includegraphics[scale=0.52]{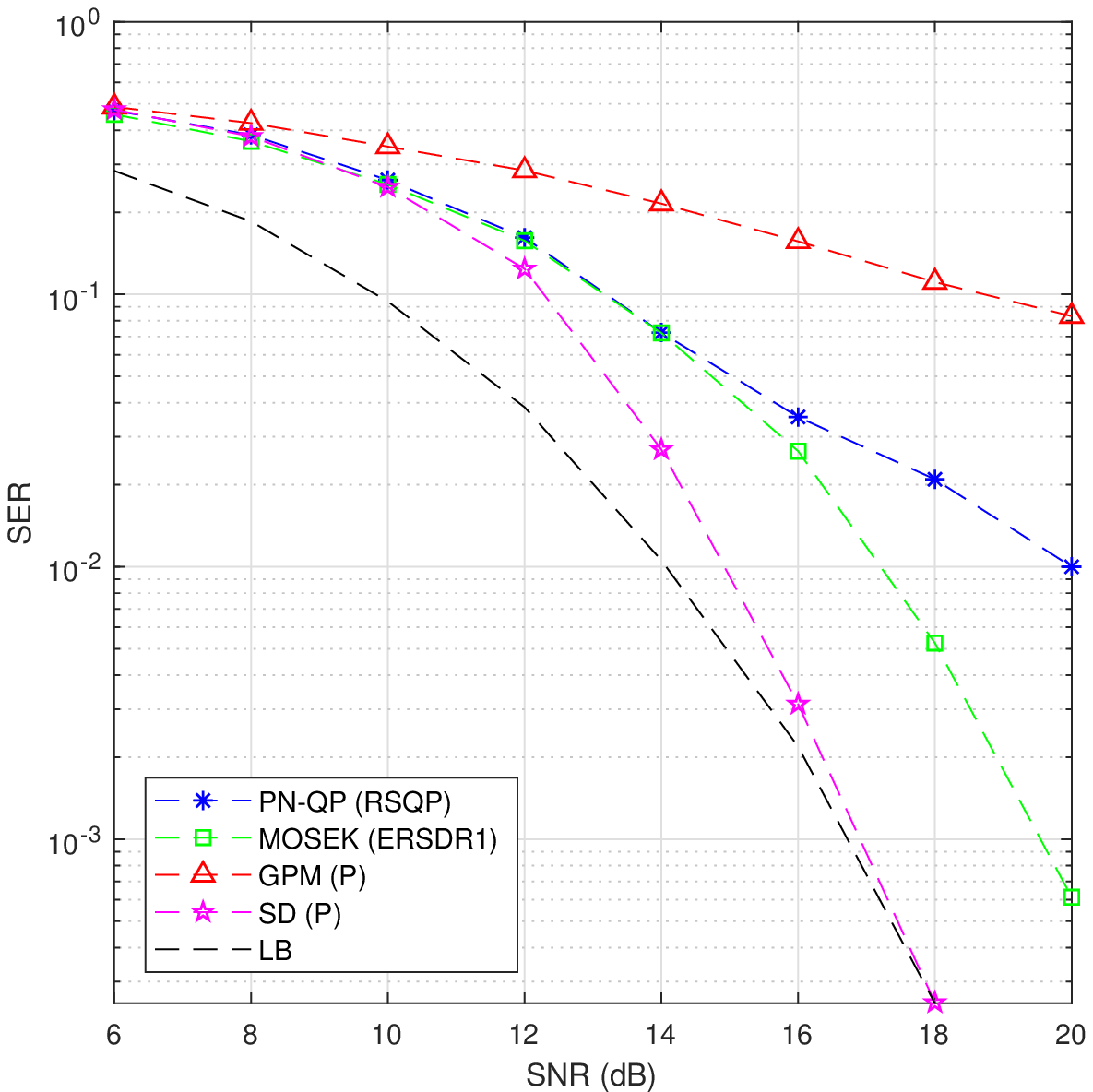}
    		\end{minipage}%
    	}%
    	\subfigure[$ (m,\,n,\,M)=(32,\,16,\,8) $.]{
    		\begin{minipage}[t]{0.48\linewidth}
    			\centering
    			\includegraphics[scale=0.52]{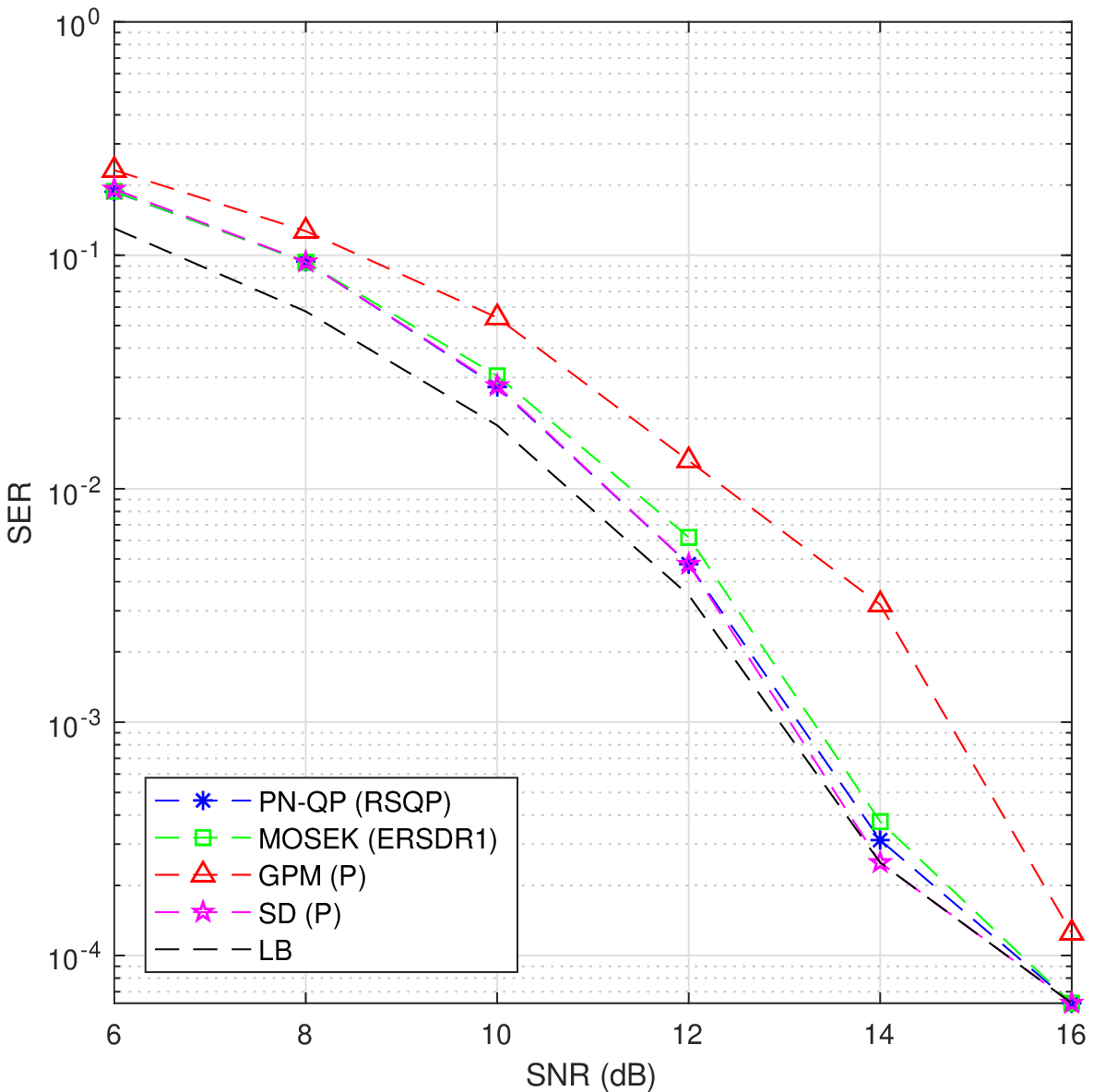}
    		\end{minipage}%
    	}
    	\centering
    	\caption{The {\rm \text{SER}} performance of the four algorithms under different {\rm \text{SNR}}s.}
    	\label{fig:ser-small}
    \end{figure}
	
	{\bf Results on Problems with $M\geqslant 8$.}
	We further compare the performance of PN-QP and GPM on $M=16$.   
	According to \cref{tab-pn-gpm}, as the SNR decreases, PN-QP provides a lower SER than GPM, implying that PN-QP achieves better detection performance. For instance, for $ (m,\,n,\,M)=(512,\,512,\,16) $ with $ {\rm SNR}=20 $ dB, PN-QP returns a solution with $ {\rm SER}=3.46\% $, whereas the SER of GPM is $ 33.67\% $.
	We also present more comparisons of the two algorithms in \cref{fig:ser-large}. 
	It can be observed from \cref{fig:ser-large} that as the SNR increases, the SER curve of PN-QP tends to coincide with the LB. In contrast, there is a large gap between the SER curve of GPM with the LB, especially in the case where $m=n$.
	\begin{table}[!htb]
		{\footnotesize
			\caption{Time and SER comparison of {\rm PN-QP} and {\rm GPM} for solving MIMO detection problems with $M=16$.}\label{tab-pn-gpm}
			\begin{center}
				\scalebox{0.9}{
					\begin{tabular}{|r|r|r|r|r|r|r|}
						\hline 
						\makecell[c]{\multirow{3}*{$ (m,\,n,\,M) $}} & \makecell[c]{\multirow{3}*{SNR}} & \multicolumn{2}{c|}{Time(s)} & \multicolumn{3}{c|}{SER($ \% $)}\\ 
						\cline{3-7}
						&  & \makecell[c]{\footnotesize{PN-QP}} & \makecell[c]{\footnotesize{GPM}} & \makecell[c]{\footnotesize{PN-QP}} & \makecell[c]{\footnotesize{GPM}} & \makecell[c]{\footnotesize{LB}} \\
						& (dB) & \makecell[c]{\footnotesize{\cref{RSQP}}} & \makecell[c]{\footnotesize{\cref{P}}} & \makecell[c]{\footnotesize{\cref{RSQP}}} & \makecell[c]{\footnotesize{\cref{P}}} &  \\
						\hline
						\makecell[c]{\multirow{4}*{$ (256,\,128,\,16) $}} & 30 & 0.203 & 0.005 & 0.00 & 0.00 & 0.00\\
						& 25 & 0.210 & 0.004 & 0.00 & 0.00 & 0.00\\
						& 20 & 0.235 & 0.004 & 0.02 & 0.02 & 0.01\\
						& 15 & 0.395 & 0.007 & 4.49 & 5.23 & 2.66\\
						\hline
						\makecell[c]{\multirow{4}*{$ (128,\,128,\,16) $}} & 30 & 0.386 & 0.007 & 0.00 & 2.81 & 0.00\\
						& 25 & 0.406 & 0.018 & 0.00 & 17.41 & 0.00\\
						& 20 & 0.549 & 0.023 & 8.42 & 34.49 & 0.74\\
						& 15 & 0.572 & 0.023 & 33.09 & 46.80 & 12.58\\
						\hline
						\makecell[c]{\multirow{4}*{$ (512,\,256,\,16) $}} & 30 & 0.432 & 0.023 & 0.00 & 0.00 & 0.00\\
						& 25 & 0.466 & 0.024 & 0.00 & 0.00 & 0.00\\
						& 20 & 0.518 & 0.026 & 0.02 & 0.02 & 0.01\\
						& 15 & 0.774 & 0.046 & 4.27 & 4.61 & 2.80\\
						\hline
						\makecell[c]{\multirow{4}*{$ (256,\,256,\,16) $}} & 30 & 0.862 & 0.047 & 0.00 & 1.60 & 0.00\\
						& 25 & 0.986 & 0.179 & 0.00 & 18.45 & 0.00\\
						& 20 & 1.513 & 0.186 & 4.80 & 34.61 & 0.65\\
						& 15 & 1.603 & 0.190 & 34.89 & 48.95 & 11.55\\
						\hline
						\makecell[c]{\multirow{4}*{$ (1024,\,512,\,16) $}} & 30 & 1.355 & 0.156 & 0.00 & 0.00 & 0.00\\
						& 25 & 1.416 & 0.155 & 0.00 & 0.00 & 0.00\\
						& 20 & 1.722 & 0.191 & 0.01 & 0.01 & 0.01\\
						& 15 & 3.273 & 0.397 & 4.48 & 4.76 & 2.85\\
						\hline
						\makecell[c]{\multirow{4}*{$ (512,\,512,\,16) $}} & 30 & 2.640 & 0.380 & 0.00 & 1.28 & 0.00\\
						& 25 & 2.980 & 2.298 & 0.00 & 19.43 & 0.00\\
						& 20 & 6.337 & 2.337 & 3.46 & 33.67 & 0.59\\
						& 15 & 6.718 & 2.390 & 33.88 & 48.36 & 12.09\\
						\hline 
				\end{tabular}}
		\end{center}}
	\end{table}
    \begin{figure}[htbp]
    	\centering
    	\subfigure[$ (m,\,n,\,M)=(512,\,512,\,8) $.]{
    		\begin{minipage}[t]{0.47\linewidth}
    			\centering
    			\includegraphics[scale=0.5]{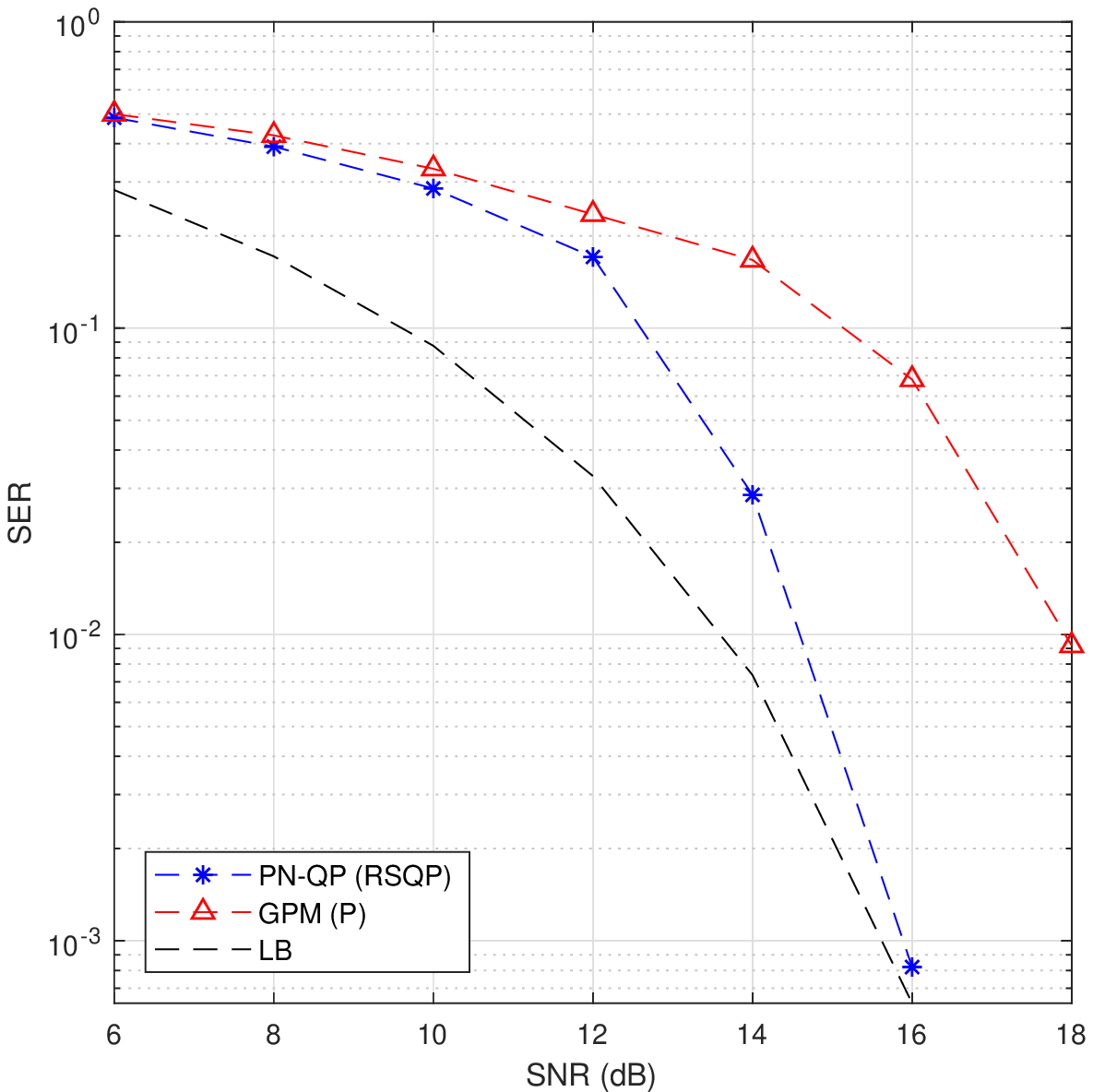}
    		\end{minipage}%
    	}%
    	\subfigure[$ (m,\,n,\,M)=(1024,\,512,\,8) $.]{
    		\begin{minipage}[t]{0.47\linewidth}
    			\centering
    			\includegraphics[scale=0.5]{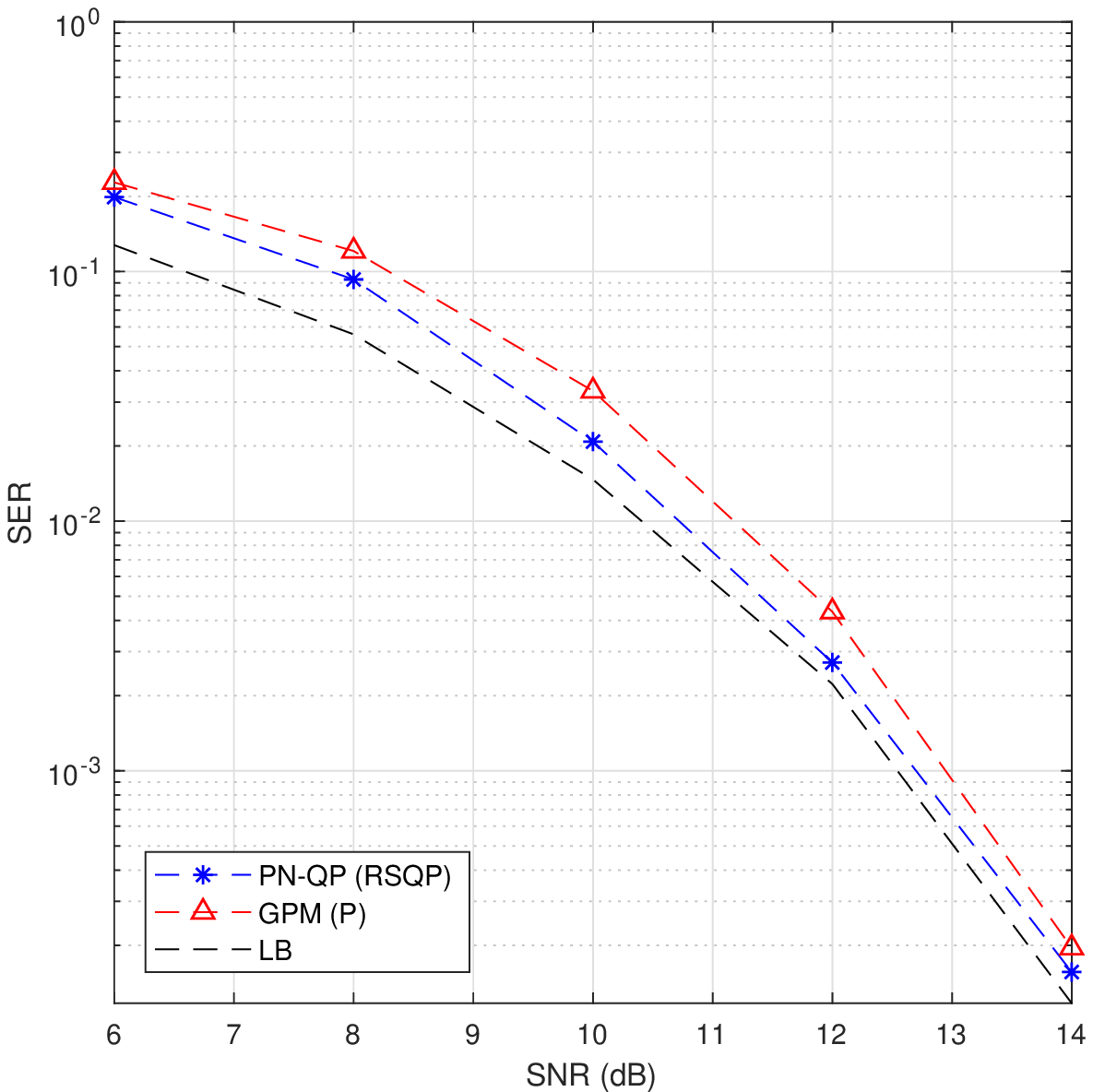}
    		\end{minipage}%
    	}\\
    	\subfigure[$ (m,\,n,\,M)=(512,\,512,\,16) $.]{
    		\begin{minipage}[t]{0.47\linewidth}
    			\centering
    			\includegraphics[scale=0.5]{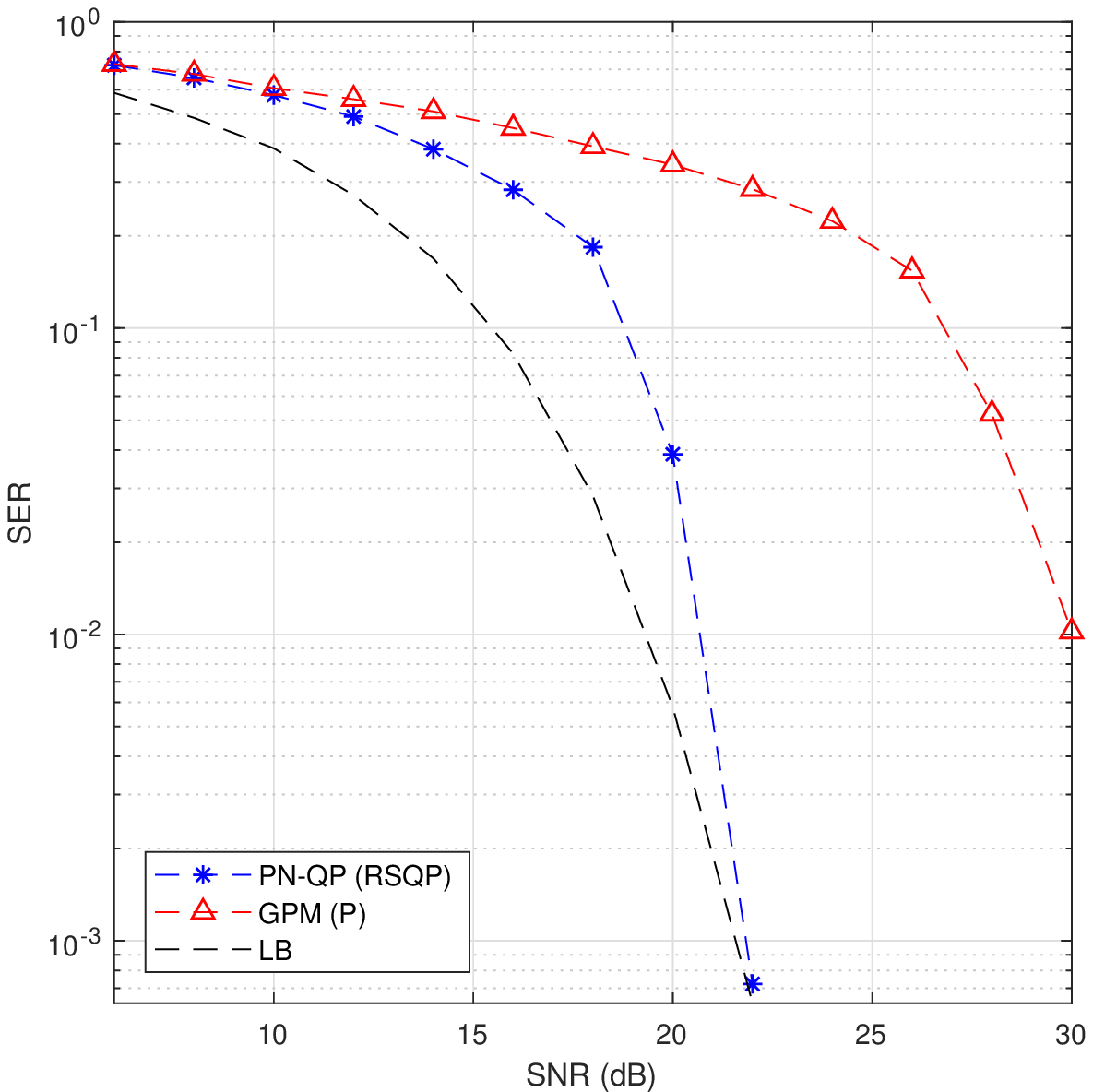}
    		\end{minipage}
    	}
    	\subfigure[$ (m,\,n,\,M)=(1024,\,512,\,16) $.]{
    		\begin{minipage}[t]{0.47\linewidth}
    			\centering
    			\includegraphics[scale=0.5]{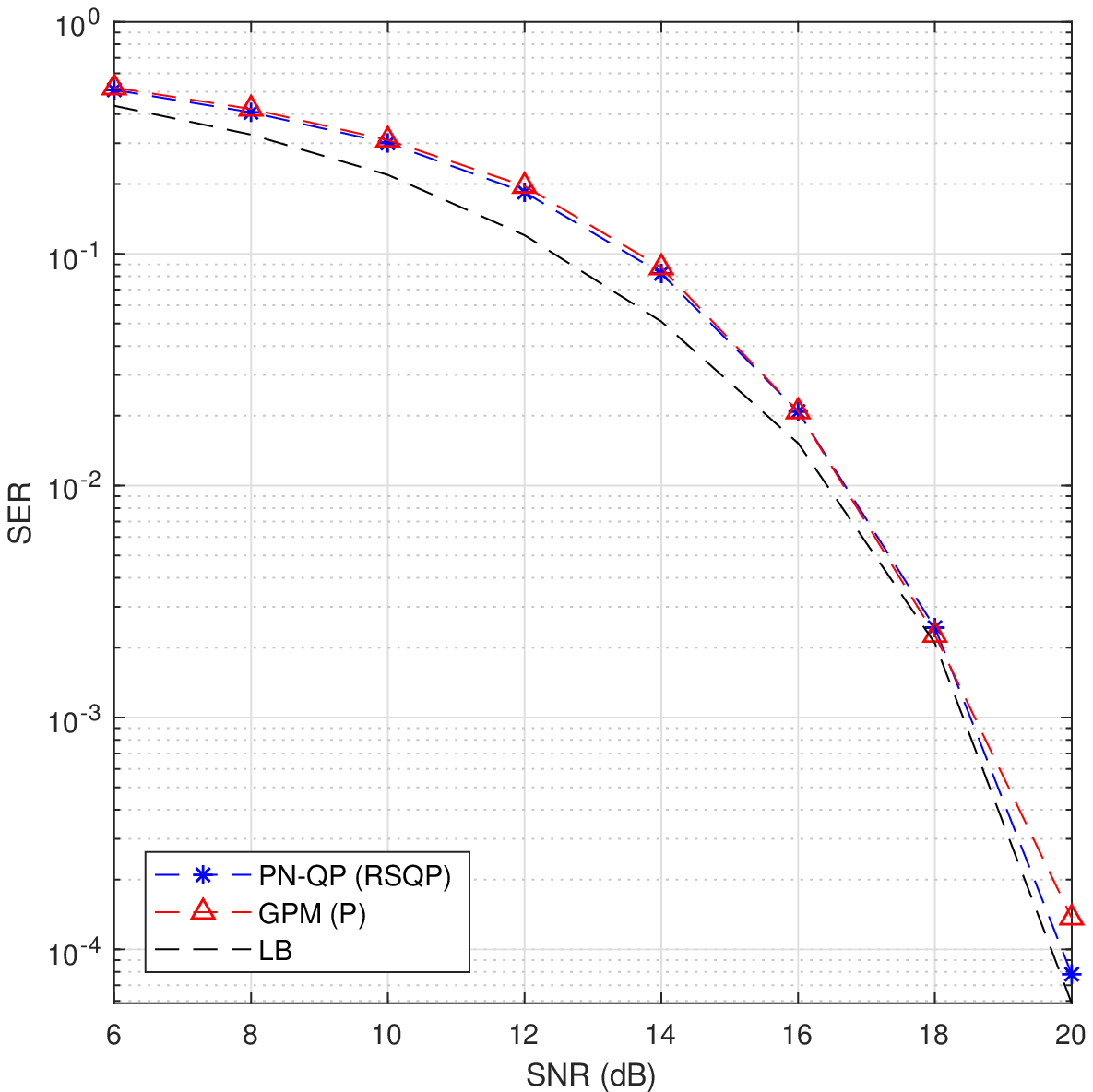}
    		\end{minipage}
    	}
    	\centering
    	\caption{The {\rm \text{SER}} performance of {\rm PN-QP} and {\rm GPM} under different {\rm \text{SNR}}s.}
    	\label{fig:ser-large}
    \end{figure}
    
    Based on the above comparisons, our proposed algorithm for solving the MIMO detection problem is more efficient 
    than existing algorithms for solving large-scale problems. Specifically, compared to MOSEK and SD, PN-QP is more efficient; compared to GPM, PN-QP achieves better detection performance. 
    To conclude, PN-QP is demonstrated to be a competitive candidate for solving the large-scale MIMO detection problem. 
    	
	\section{Conclusions}\label{sec-conclusions}
	In this paper, we proposed an efficient algorithm called PN-QP for solving the \emph{large-scale} MIMO detection problem, motivated by the massive MIMO technology. The proposed algorithm is essentially a quadratic penalty method applied to solve an SQP relaxation, i.e., problem \cref{RSQP}, of the original problem. Two key features of the proposed algorithm, which make it particularly suitable to solve the large-scale problems, are: (i) it is based on the relaxation problem \cref{RSQP}, whose numbers of variables and constraints are significantly less than those of the SDRs; and (ii) our proposed algorithm is custom-designed to identify the support set of the optimal solution by judiciously exploiting the special structure of the problem, instead of finding the solution itself, which thus substantially reduces the computational complexity of the proposed algorithm. The above two reasons lead to the better numerical performance of the proposed algorithm. In particular, our extensive simulation results show that our proposed algorithm compares favorably with the state-of-the-art algorithms (including SD and SDR based approaches) for solving the MIMO detection problem. When applied to solve large-scale problems, our proposed algorithm achieves significantly better detection performance than GPM.
	
	\appendix
	\section{Proof of \cref{thm-sqp-sqp2}}\label{app-thm-sqp-sqp2}
	We need the following results to prove \cref{thm-sqp-sqp2}.
	\begin{proposition}\label{prop-property}
		Let $Q$, $\widehat Q$, and $G$ be defined in \cref{def-Q}, \cref{DINGYI}, and \cref{Gw}, respectively. We have
		\bit
		\item [{\rm (i)}]  $q_{jj} = \sum_{k = 1}^m|h_{kj}|^2$ for $j = 1, \ldots,n$;
		\item [{\rm (ii)}] $ \hat q_{kj} = \hat q_{jk} = \hat q_{(n+j)(n+k)} = \hat q_{(n+k)(n+j)} $, $ \hat q_{(n+j)k} = \hat q_{k(n+j)} = -\hat q_{(n+k)j} = -\hat q_{j(n+k)} $, $\hat q_{jj} = \hat q_{(n+j)(n+j)} = q_{jj}$, and $\hat q_{(n+j)j} = \hat q_{j(n+j)} = 0$ for $j,\ k = 1, \ldots,n$;
		\item [{\rm (iii)}]  $ S_{jk}=\hat{q}_{jk}\alpha\alpha^\top+\hat{q}_{(n+j)k}\beta\alpha^\top+\hat{q}_{j(n+k)}\alpha\beta^\top+\hat{q}_{(n+j)(n+k)}\beta\beta^\top $
		and $S_{jk} = S_{kj}^\top$ for $j,\ k = 1, \ldots,n$; and
		\item [{\rm (iv)}] $\diag (S_{jj}) = q_{jj}\boldsymbol{e}$ for $j = 1, \ldots,n$.
		\eit	
	\end{proposition}
	\begin{proof}
		(i) By the definition of $Q$ in \cref{def-Q}, there is
		\[
		q_{jj}= \sum_{k = 1}^m h_{kj}^\dagger  h_{kj} = \sum_{k = 1}^m |h_{kj}|^2,
		\]
		which gives (i).
		
		(ii)  By the definitions of $\widehat Q$ in \cref{DINGYI} and $ Q $ in \cref{def-Q}, the first three results in (ii) hold naturally. Note that due to (i), there is $\diag (\IM(Q))=\textbf{0},$ implying the fourth result in (ii).
		
		(iii) Let $ P $ have the partition as $ P=\begin{bmatrix}P_1 & \cdots & P_n\end{bmatrix} $, where $ P_j \in \mathbb{R}^{2n\times M} $ takes the following form: 
		\begin{equation}\label{Pj}
		\begin{aligned}
		\begin{array}{ccccccccccc}
		P_j^\top=[\textbf{0} & \cdots & \textbf{0} & \alpha & \textbf{0} & \cdots & \textbf{0} & \beta & \textbf{0} & \cdots & \textbf{0}]^\top.\\
		&&&\uparrow&&&&\uparrow&&&
		\\
		&\multicolumn{4}{r}{\text{\footnotesize{the $j$-th block}}}&\multicolumn{5}{c}{\text{\footnotesize{the $(n+j)$-th block}}}&
		\end{array}
		\end{aligned}
		\end{equation}
		By the definition and partition of $ G $ in \cref{Gw,G}, we have
		\[
		\begin{aligned}
		G&=P^\top \widehat{Q}P=\begin{bmatrix}
		P_1 & \cdots & P_n
		\end{bmatrix}^\top \widehat{Q}\begin{bmatrix}
		P_1 & \cdots & P_n
		\end{bmatrix}\\
		&=\begin{bmatrix}
		P_1^\top\widehat{Q}P_1 & \cdots & P_1^\top\widehat{Q}P_n\\
		\vdots & \ddots & \vdots\\
		P_n^\top\widehat{Q}P_1 & \cdots & P_n^\top\widehat{Q}P_n
		\end{bmatrix}=\begin{bmatrix}
		S_{11} & \cdots & S_{1n}\\
		\vdots & \ddots & \vdots\\
		S_{n1} & \cdots & S_{nn}
		\end{bmatrix},
		\end{aligned}
		\]
		which, together with \cref{Pj,DINGYI}, further implies
		\[
		\begin{aligned}
		S_{jk} &=P_j^\top \widehat Q P_k\\
		&=\begin{bmatrix}
		\hat q_{j1}\alpha +\hat q_{(n+j)1}\beta & \cdots & 
		\hat q_{j(2n)}\alpha +\hat q_{(n+j)(2n)}\beta
		\end{bmatrix} P_k\\
		&=\hat{q}_{jk}\alpha\alpha^\top+\hat{q}_{(n+j)k}\beta\alpha^\top+\hat{q}_{j(n+k)}\alpha\beta^\top+\hat{q}_{(n+j)(n+k)}\beta\beta^\top.
		\end{aligned}
		\]
		The proof of the first result in (iii) is finished.
		With the first result in (iii), as well as the first two results in (ii), there is
		\[
		\begin{aligned}
		S_{jk}  &=\hat{q}_{jk}\alpha\alpha^\top+\hat{q}_{(n+j)k}\beta\alpha^\top+\hat{q}_{j(n+k)}\alpha\beta^\top+\hat{q}_{(n+j)(n+k)}\beta\beta^\top\\
		&=\hat{q}_{kj}\alpha\alpha^\top+\hat{q}_{k(n+j)}\beta\alpha^\top+\hat{q}_{(n+k)j}\alpha\beta^\top+\hat{q}_{(n+k)(n+j)}\beta\beta^\top\\
		&=(\hat{q}_{kj}\alpha\alpha^\top+\hat{q}_{k(n+j)}\alpha\beta^\top+\hat{q}_{(n+k)j}\beta\alpha^\top+\hat{q}_{(n+k)(n+j)}\beta\beta^\top)^\top\\
		&=S_{kj}^\top.\\
		\end{aligned}
		\]
		We get the second result in (iii).
		
		(iv) Due to (iii), there is
		\[
		\begin{aligned}
		S_{jj}&=\hat{q}_{jj}\alpha\alpha^\top+\hat{q}_{(n+j)j}\beta\alpha^\top+\hat{q}_{j(n+j)}\alpha\beta^\top+\hat{q}_{(n+j)(n+j)}\beta\beta^\top\\
		&= q_{jj}(\alpha\alpha^\top+\beta\beta^\top).
		\end{aligned}
		\]
		Recall the definitions of $\alpha$, $\beta$ in \cref{def-alphabeta} and $\theta_k$ in \cref{YJIHE}. The $k$-th diagonal entry of $S_{jj}$ takes the following form:
		\[
		q_{jj}(\cos^2\left( \theta_k\right) + \sin^2\left( \theta_k\right) ) = q_{jj},\ k = 1,\ldots,M.
		\]
		Therefore, $\diag (S_{jj})=q_{jj}\boldsymbol{e}$. The proof is finished.
	\end{proof}
	
	\begin{proposition} \label{prop-diag-cons} Under the constraints in problem \cref{SQP2}, there is
		\begin{equation}\label{eq-sum}
		\sum_{j = 1}^n \bar t_j^\top S_{jj} \bar t_j= \|H\|_2^2.
		\end{equation}	
	\end{proposition}
	\begin{proof}
		Under the constraints in problem \cref{SQP2}, for each $\bar t_j$, there exists $l_j\in\{1,\ldots,M\}$, such that $\bar t_j = e_{l_j}$, where $e_{l_j}$ denotes the $l_j$-th column in the identity matrix $I\in\mathbb R^{M\times M}$. As a result, with (i) and (iv) in \cref{prop-property}, there is
		$$
		\sum_{j = 1}^n \bar t_j^\top S_{jj} \bar t_j=\sum_{j=1}^n e_{l_j}^\top S_{jj}e_{l_j} 
		=  \sum_{j=1}^n q_{jj} 
		=  \sum_{j = 1}^n \sum_{k=1}^m |h_{kj}|^2 
		=  \|H\|_2^2,
		$$
		which gives \cref{eq-sum}. The proof is finished.
	\end{proof}
	
	Now we are ready to prove \cref{thm-sqp-sqp2}.
	\begin{proof}
		Using \cref{prop-diag-cons}, we have 
		\[
		h(t)-f(t) = \sum_{j = 1}^n\bar t_j^\top S_{jj}\bar t_j = \|H\|_2^2.
		\]
		This, together with the fact that the constraints of problems \cref{SQP1} and \cref{SQP2} are the same, implies that problems \cref{SQP1} and \cref{SQP2} are equivalent.
	\end{proof}

	\section*{Acknowledgments}
	We would like to thank the associate editor Professor William Hager for handling our submission as well as the two anonymous reviewers for their insightful comments. We would also like to thank Dr. Huikang Liu for kindly sharing  the GPM code with us.
	
	\bibliographystyle{siamplain}
	\bibliography{references}
\end{document}